\documentclass[11pt]{amsart}
\usepackage{amsmath, amsfonts, amsthm, amssymb, multicol, mathtools, dsfont, verbatim}
\usepackage{graphicx}
\usepackage{float, hyperref}
\usepackage{scalerel,stackengine, subcaption}
\usepackage[usenames,dvipsnames]{xcolor}
\usepackage{enumitem}
\usepackage{pgfplots}
\usepgfplotslibrary{fillbetween}
\pgfplotsset{width=10cm,compat=1.9}
\usepackage[square,sort,comma,numbers]{natbib}
\newcommand{\what}{\widehat}

\allowdisplaybreaks

\usepackage{fancyhdr}

\makeatletter
\def\@setauthors{%
  \begingroup
  \def\thanks{\protect\thanks@warning}%
  \trivlist
  \centering\footnotesize \@topsep30\p@\relax
  \advance\@topsep by -\baselineskip
  \item\relax
  \author@andify\authors
  \def\\{\protect\linebreak}

  \normalsize\lowercase{\authors}%
  
	\ifx\@empty\contribs
  \else
    ,\penalty-3 \space \@setcontribs
    \@closetoccontribs
  \fi
  \endtrivlist
  \endgroup
}
\def\@settitle{\begin{center}
\LARGE\lowercase{\@title}
  \end{center}%
}
\makeatother
\newcommand{\authoremail}[1]{\email{\href{mailto:#1}{\color{lightblue}{#1}}}}
\newcommand{\authoraddress}[1]{\address{\normalfont{#1}}}

\definecolor{lightblue}{HTML}{2B77A4}
\definecolor{darkred}{HTML}{9E0D0D}
\hypersetup{
	colorlinks=true,
	linkcolor=darkred,
	urlcolor=darkred,
	citecolor=lightblue
}
\numberwithin{equation}{section}

\renewcommand{\epsilon}{\varepsilon}

\newcommand{\rd}{\mathbb{R}^d}

\newcommand{\M}{\mathcal{M}}

\renewcommand{\ge}{\geqslant}
\renewcommand{\le}{\leqslant}
\renewcommand{\geq}{\geqslant}
\renewcommand{\leq}{\leqslant}

\newcommand{\ubd}{\overline{\dim}_{\textup{B}}}

\newcommand{\hd}{\dim_{\textup{H}}}

\newcommand{\fs}{\dim^\theta_{\mathrm{F}}}
\newcommand{\fd}{\dim_{\mathrm{F}}}
\newcommand{\sd}{\dim_{\mathrm{S}}}
\newcommand{\J}{\mathcal{J}}

\usepackage{xcolor}
\usepackage{pgfplots}
\pgfplotsset{compat=1.18}

\usepackage{a4wide}

\usepackage[a4paper, margin=2.3cm]{geometry}
\usepackage{amsmath,amssymb,amsthm, comment, enumitem, array}
\usepackage{color}
\usepackage{svg}
\usepackage{parskip}
\usepackage{hyperref}
\usepackage{parskip}
\usepackage{setspace}
\usepackage{graphicx}
\usepackage{mathtools}
\usepackage[thinc]{esdiff}
\usepackage[euler]{textgreek}

\usepackage{babel}

\usepackage[capitalise, nameinlink]{cleveref}
\crefname{equation}{}{}
\crefname{figure}{{\sc Figure}}{{\sc Figure}}
\crefname{subsection}{Subsection}{Subsections}

\newtheorem{theorem}{Theorem}[section]
\newtheorem{proposition}[theorem]{Proposition}
\newtheorem{remark}[theorem]{Remark}
\newtheorem{corollary}[theorem]{Corollary}
\newtheorem{lemma}[theorem]{Lemma}

\newtheorem{conjecture}[theorem]{Conjecture}
\newtheorem*{definition*}{Definition}

\newcommand{\spt}{\operatorname{supp}}

\usepackage{amsmath,amssymb,amsthm,mathtools}
\usepackage{enumitem}
\usepackage{hyperref}
\usepackage{parskip}
\hypersetup{colorlinks=true,linkcolor=blue,citecolor=blue,urlcolor=blue}

\theoremstyle{plain}

\theoremstyle{remark}

\newcommand{\R}{\mathbb{R}}
\newcommand{\Rd}{\mathbb{R}^d}
\newcommand{\Rdpo}{\mathbb{R}^{d+1}}
\newcommand{\supp}{\operatorname{supp}}

\newcommand{\cJ}{\mathcal{J}}
\newcommand{\cL}{\mathcal{L}}

\newcommand{\ang}[1]{\langle #1\rangle}

\title{On Fourier decay and the distance set problem}

\author{\large{Jonathan M. Fraser}}
\authoraddress{Jonathan M. Fraser, University of St Andrews, Scotland}
\authoremail{jmf32@st-andrews.ac.uk}
\thanks{JMF was  financially supported by an \emph{EPSRC Open Fellowship} (EP/Z533440/1) and a  \emph{Leverhulme Trust Research Project Grant} (RPG-2023-281).}

 \author{\large{Thang Pham}}
\authoraddress{Thang Pham, Institute for Mathematics and Interdisciplinary Sciences, Xidian University, China}
\authoremail{thangphammath@xidian.edu.cn}

\begin{document}


\maketitle
\thispagestyle{empty}

\begin{abstract}
We study the Falconer distance set problem in Euclidean space and obtain  improved dimensional estimates under natural  Fourier analytic assumptions cast in terms of the  Fourier dimension and spectrum.  Interestingly, under  reasonably mild  assumptions, we are able to beat the $d/2$ dimension threshold in dimensions $d \geq 5$. For example, we show that (in any ambient spatial dimension $d$) a Borel set with Fourier  dimension at least $2$ has a distance set of full Hausdorff dimension.  We also show that (in any ambient spatial dimension $d$) a Borel set with Fourier  \emph{spectrum} at least $d/4+1$ at $\theta=1/2$ has a distance set of full Hausdorff dimension.  In particular, this can hold for sets with Fourier dimension zero (provided $d \geq 4$). We also consider pinned variants of these problems and construct examples that demonstrate the sharpness (or near sharpness) of our results. 
\\ \\ 
\emph{Mathematics Subject Classification 2020}. primary: 28A80, 42B10; secondary: 28A75, 28A78.
\\
\emph{Key words and phrases}:  distance set problem, Hausdorff dimension, Fourier dimension, Fourier spectrum.
\end{abstract}

\tableofcontents

\section{Introduction}

\subsection{The distance set problem and summary of our results}

The \emph{distance set problem} is a famous and important problem in geometric measure theory.  It was introduced in an influential paper of Falconer from 1985 \cite{distance}, but can be viewed as a continuous analogue of the famous distance problem of Erd\H{o}s which asks how many distinct distances are guaranteed to  be determined by a set of $N$ points in the plane.  

Given a Borel set $E \subseteq \Rd$, the \emph{distance set} is defined by
\[
D(E) = \{|x-y| : x, y \in E\},
\]
that is, it is the set of distances realised between points in $E$. The distance set problem is then to determine how large $D(E)$ must be, given certain geometric constraints on $E$.  Here we interpret `large' in terms of Hausdorff dimension, and the geometric constraint on $E$ might simply be that it is itself large in terms of dimension.  For example, one of the most well-known versions of the problem is the conjecture that 
\begin{equation} \label{distanceproblem}
\hd E \geq d/2 \ \Rightarrow \  \hd D(E) = 1.
\end{equation}
This conjecture is open for all $d \geq 2$ despite sustained interest over several decades and much recent progress, see \cite{1/4even,1/4+,yumeng,guth,keletishmerkin,shmerkinwang}. It is known, and important to keep in mind for this paper, that  the threshold $d/2$ in (\ref{distanceproblem}) cannot be improved, see \cite{distance}.  That is, there exist (compact) $E \subseteq \rd$ with $\hd E$  arbitrarily close to (but smaller than) $d/2$, but yet $\hd D(E)<1$.  The sets used to demonstrate the $d/2$ threshold are very rigid and possess a lot of arithmetic resonance.  These features mean that the Fourier analytic behaviour of measures supported on these examples must be very bad.  This perhaps suggests that, under mild Fourier analytic assumptions, the threshold $d/2$ might be lowered.  This is our main objective in this paper.  In the opposite direction from the rigid examples described above, it is natural    to consider sets with optimal Fourier decay, namely, Salem sets (see the next section for a precise definition).  Mattila proved in  \cite{mattiladistance} that the distance set problem \eqref{distanceproblem} can be resolved  for Salem sets in the sense that if $E$ is Salem, then the conjectured implication
\[
\hd E \geq d/2 \Rightarrow \hd D(E) = 1
\]
indeed holds. We show that this result can be dramatically improved in high dimensions by proving that if $E$ is Salem, then
\[
\hd E \geq  2 \Rightarrow \hd D(E) = 1,
\]
see {\bf Theorem \ref{thm-theta-zero}} and {\bf Corollaries \ref{cor2} and \ref{cor3}}. In particular, our result beats the $d/2$ threshold for $d \geq 5$.  This result  is a special case of more general results which require much weaker Fourier analytic assumptions.  These assumptions are cast in terms of the Fourier spectrum and give concrete, checkable, and relatively mild conditions ensuring that the distance set of a given set is of full Hausdorff dimension or even has positive 1-dimensional Lebesgue measure, see  {\bf Corollaries \ref{cormain} and \ref{cor2}}.  

We also consider the more general problem where one studies the distances realised between two given sets, see our {\bf main Theorem \ref{thm-ref}} as well as {\bf   Corollary \ref{cor1/2}}, and  the problem where one fixes a `pin' and considers only distances realised between the pin and other points.  More precisely, given Borel sets $E,F \subseteq \rd$, we define the \emph{distance set}
\[
D(E,F) = \{ |x-y| : x \in E, y \in F\}.
\]
Then $D(E,E) = D(E)$ and 
\[
D(\{x\}, E) =: D_x(E)
\]
is the \emph{pinned} distance set for $x \in \rd$.  If $x \in E$, then $D_x(E) \subseteq D(E)$ and \emph{a priori} the pinned distance set should be much smaller than the full distance set.  In particular, there is interest in proving the stronger conclusion that, for example, there exists $x \in E$ such that $\hd D_x(E) = 1$ or that $\hd D_x(E) = 1$ for `many' pins $x \in E$. Our results are of this type and we are often able to find a set of pins $E' \subseteq E$ of positive Hausdorff dimension where the desired conclusion holds.

\subsection{Hausdorff dimension, the Fourier spectrum and previous work}

In this section we briefly recall and discuss our key notions of dimension, including the Hausdorff dimension, Fourier dimension, and the Fourier spectrum.  

In order to define the Hausdorff dimension, first define the  \emph{$s$-energy} of a finite Borel measure $\mu$ on $\rd$  by
\begin{equation*}
		I_{s}(\mu) = \iint |x-y|^{-s}\,d\mu(x)\,d\mu(y).
\end{equation*}
The \emph{Hausdorff dimension} of a Borel set $E$ may then be defined by
\begin{equation*}
		\hd E = \sup\{ s\geq0 : \exists \mu\in\M(E) : I_{s}(\mu)<\infty \}
\end{equation*}
where $\M(E)$ denotes the set of finite, non-zero Borel measures with support contained in $E$. Deriving the Hausdorff dimension via this alternative definition is often referred to as the \emph{potential theoretic method};  see \cite{falconer,mattila} for more details.  These energy integrals can be expressed in terms of the Fourier transform, and this connection opens up a rich interplay between fractal geometry and Fourier analysis. More precisely, the $s$-energy of $\mu\in\M(\rd)$, for $s\in(0,d)$, satisfies
\begin{equation} \label{energyequiv}
		I_{s}(\mu) =C({d,s}) \int \big| \widehat{\mu}(z) \big|^2 |z|^{s-d}\,dz
\end{equation}
where $C(d,s)$ is a constant depending only on $s$ and $d$. This relationship between the Hausdorff dimension of sets and the Fourier transform of measures they support motivates the definition of the \emph{Fourier dimension} of a finite Borel measure.  This captures the optimal decay rate of the Fourier transform and is defined by
\begin{equation*}
		\fd \mu = \sup\big\{ s\geq0 : \sup_z\big| \widehat{\mu}(z) \big|^2  |z|^{s} < \infty \big\}.
\end{equation*}
The \emph{Fourier dimension} of a Borel set  $E\subseteq\rd$ is then defined by 
\begin{equation*}
		\fd E = \sup\big\{ \min\{ \fd \mu,  d\} : \mu\in\M(E) \big\}.
\end{equation*}
By leveraging  a uniform decay estimate for the Fourier transform in the Fourier analytic expression for energy, one can easily check that
\[
0\leq\fd E\leq\hd E \leq d.
\]
Moreover,   these inequalities can be strict in any combination. Sets for which the Fourier and Hausdorff dimensions coincide are called \emph{Salem sets} and such sets can be thought of as having `optimal Fourier analytic properties'.

The Fourier spectrum was introduced by   Fraser  in \cite{JMF} to interpolate between the Fourier   and Hausdorff dimensions. As observed above, the Fourier dimension concerns $L^\infty$ decay of the Fourier transform, and the Hausdorff dimension concerns  $L^2$ decay.   If the Fourier  and Hausdorff dimensions are distinct, then it is natural  to try to understand the decay in an appropriate $L^p$ sense and this is what the Fourier spectrum does.

In order to define the Fourier spectrum,  first define \emph{$(s,\theta)$-energies} of  $\mu$  by
\begin{equation*}
	\J_{s,\theta}(\mu) = \bigg( \int_{\rd} \big| \widehat{\mu}(z) \big|^{\frac{2}{\theta}}|z|^{\frac{s}{\theta}-d}\,dz \bigg)^\theta,
\end{equation*}
for $\theta \in (0,1]$ 
and, for $\theta = 0$, by
\begin{equation*}
	\J_{s,0}(\mu) = \sup_{z\in\rd} \big| \widehat{\mu}(z) \big|^2|z|^{s}.
\end{equation*}
Then the \emph{Fourier spectrum} of $\mu$ at $\theta \in [0,1]$ is defined by
\begin{equation*}
	\fs \mu = \sup\{ s \geq 0 : \J_{s,\theta}(\mu)<\infty \},
\end{equation*}
and, for each $\theta\in[0,1]$, $\fd \mu\leq\fs \mu\leq\sd \mu \leq \hd \mu$, with equality on the left if $\theta= 0$ and equality on the right if $\theta = 1$. Here $\sd \mu$ denotes the Sobolev (or energy) dimension of $\mu$ and $\hd \mu$ denotes the Hausdorff dimension of $\mu$ defined by
\[
\hd \mu = \inf\{ \hd E : \mu(E)>0\}.
\]
We write $\spt \mu$ for the support of $\mu$ and then
\[
\hd \mu \leq \hd \spt \mu.
\]
As a function of $\theta$, $\fs\mu$ is concave and continuous for $\theta\in(0,1]$ by \cite[Theorem~1.1]{JMF} and, in addition,  continuous at $\theta=0$ provided $\mu$ is compactly supported by \cite[Theorem~1.3]{JMF}. For  a Borel set $E\subseteq\rd$, the \emph{Fourier spectrum} is defined by
\begin{equation*}
	\fs  E = \sup\big\{ \min\{\fs \mu, d\} : \mu\in\M(E) \big\}.
\end{equation*}
Then,  for all $\theta\in[0,1]$, $\fd E\leq\fs  E \leq \hd E$, with equality on the left if $\theta= 0$ and equality on the right if $\theta = 1$. Moreover,  $\fs E$ is continuous for all $\theta\in[0,1]$ by \cite[Theorem~1.5]{JMF}.  Unlike for measures, $\fs E$ need not be concave.

The main purpose of the Fourier spectrum is to provide a more nuanced quantitative description of the Fourier analytic properties of sets and measures and to leverage this description to make progress on various problems where there is an interplay between Fourier analysis and fractal geometry.  The Fourier spectrum has already found numerous applications of this type, often leading to  stronger results than one can get by appealing to the Fourier dimension alone.  These applications include new Hausdorff dimension estimates in the Falconer distance problem \cite[Section 7]{JMF} and the celebrated restriction problem in harmonic analysis \cite{restriction}. In particular, the following was obtained in \cite{JMF}.  Here and throughout, we write $\mathcal{L}^1$ for 1-dimensional Lebesgue measure on $\mathbb{R}$ and note that $\mathcal{L}^1(E)>0$ is a strictly stronger conclusion than $\hd E = 1$. 

\begin{theorem} [\cite{JMF}, Thm.~7.3]
\label{jfthm}
Suppose $\mu_1$ and $\mu_2$ are finite Borel measures on $\mathbb{R}^d$ with
\[
\mathcal{J}_{u,\theta}(\mu_1)<\infty
\qquad\text{and}\qquad
\mathcal{J}_{v,1-\theta}(\mu_2)<\infty
\]
for some $u,v\ge 0$ and $\theta\in(0,1]$. If $u+v\ge d$, then
\[
\mathcal{L}^1\!\big(D(\spt \mu_1,\spt \mu_2)\big)>0,
\]
and, if $u+v<d$, then
\[
\hd D(\spt\mu_1,\spt \mu_2)\ \ge\ 1-d+u+v.
\]
\end{theorem}

This extends Mattila’s result, which relies only on Fourier dimension. Indeed, a simple consequence of the above is that 
\begin{equation} \label{fsat1/2}
\dim_{\textup{F}}^{1/2} E \geq d/2 \Rightarrow \hd D(E) = 1
\end{equation}
 holds. In particular, this solves the distance set conjecture at the $d/2$ threshold for a much larger class of sets than Salem sets. Indeed, a set $E$ may have Fourier dimension 0, but still have  $\dim_{\textup{F}}^{1/2} E = d/2 $.  In this paper we significantly improve the above theorem, for example, replacing the $d/2$ threshold in \eqref{fsat1/2} with $d/4+1$, see {\bf Corollary \ref{cor3}}.

\section{Main results} \label{sec:results}

In this section we state our main results.  First we give a general result and then we state some  corollaries which may be simpler to digest and to compare with the literature.

\begin{theorem}\label{thm-ref}
Let \(d\ge 2\), and let \(\mu_1,\mu_2\) be probability measures on \(\Rd\) with
compact supports. Let $\theta_1,\theta_2 \in [0,1]$ and write
\[
u = \dim_\textup{F}^{\theta_1} \mu_1
\]
and
\[
a=\min\!\left\{\frac d2,\,\dim_\textup{F}^{\theta_2} \mu_2, \, \frac{\dim_\textup{F}^{\theta_2} \mu_2}{2\theta_2}\right\}.
\]
Assume $u \leq d$ and define
\[
\beta(u):=
\begin{cases}
\dfrac{u+2\theta_1 a-d\theta_1}{2},
& 0\le u<d\theta_1,\\[2ex]
\dfrac{ua}{d},
& d\theta_1\le u\le d.
\end{cases}
\]
If \(\beta(u)>0\), then
\[
\hd D_x(\spt\mu_2)\ge \min\{1,\beta(u)\},
\]
for $\mu_1$ almost all $x$ and, if \(\beta(u)>1\), then
\[
\cL^1\bigl(D_x(\spt\mu_2)\bigr)>0
\]
for $\mu_1$ almost all $x$. In particular,
\[
\hd D(\spt\mu_1,\spt\mu_2)\ge \min\{1,\beta(u)\},
\]
and if \(\beta(u)>1\), then
\[
\cL^1\bigl(D(\spt\mu_1,\spt\mu_2)\bigr)>0.
\]
\end{theorem}
\begin{remark}
By convention, for $\theta_2=0$ we interpret
\[
\frac{\dim_{\mathrm F}^{\theta_2}\mu_2}{2\theta_2}=+\infty.
\]
Hence, in this case
\[
a=\min\left\{\frac d2,\dim_{\mathrm F}^{0}\mu_2, \infty\right\}
=\min\left\{\frac d2,\dim_{\mathrm F}\mu_2\right\}.
\]
\end{remark}

We prove Theorem \ref{thm-ref} in Section \ref{sec:proof}. The assumption that $u \leq d$ is for convenience and if $u >d$ then we can take $\beta(u) = a$. The case $u>d$ is not so interesting since  by \cite[Theorem 5.4]{mattila} the $d$-dimensional Lebesgue measure of $\spt \mu_1$ is positive and then it is easy to show that (for example) the 1-dimensional Lebesgue measure of $D(\spt \mu_1, \spt \mu_2)$ is also positive.  

The special case of the above when $\theta_1=\theta_2=1/2$ is notable. Partly because it makes it easy to compare with Theorem \ref{jfthm}, but also because this is the regime in which our proof approach works most directly.  We will also scrutinise the `near sharpness' of this result in Section \ref{sec:examples}.   We  assume $\dim_\mathrm{F}^{1/2} \mu_1 , \dim_\mathrm{F}^{1/2} \mu_2\leq d/2$ in the following since this is the case of interest and it also simplifies the exposition. 

\begin{corollary} \label{cor1/2}
    Let \(d\ge 2\), and let \(\mu_1,\mu_2\) be probability measures on \(\Rd\) with
compact supports. Further,   assume both $\dim_\mathrm{F}^{1/2} \mu_1 \leq d/2$ and $\dim_\mathrm{F}^{1/2} \mu_2\leq d/2$.  Then
\[
\hd D_x(\spt\mu_2)\ge \min\left\{ 1, \, \frac{\dim_\mathrm{F}^{1/2} \mu_1+\dim_\mathrm{F}^{1/2} \mu_2}{2}- \frac{d}{4}\right\}
\]
for $\mu_1$ almost all $x$.
\end{corollary}

Note that the previous corollary gives a (trivial)  lower bound of zero in the case $\dim_\mathrm{F}^{1/2} \mu_1+\dim_\mathrm{F}^{1/2} \mu_2 = d/2$.    Perhaps surprisingly, this is sharp, see  Proposition \ref{ex1} where we construct measures with  dimensions satisfying this equality for which there is only a single distance realised in $D(\spt\mu_1,\spt\mu_2)$.

Theorem \ref{thm-ref} works especially well at $\theta=1/2$ and generally for larger $\theta$.  In another direction, we obtain the following theorem which works especially well at $\theta=0$ and generally for small $\theta$.  

 \begin{theorem}\label{thm-theta-zero}
Let \(d\ge 2\), and let \(\mu\) be a compactly supported probability
measure on \(\Rd\).
\begin{enumerate}
    \item For \(\mu\)-almost all \(x\in\Rd\),
    \[
    \hd D_x(\spt \mu)\ge \min\left\{\frac{\fd\mu}{2},1\right\}.
    \]
    Moreover, if \(\fd\mu>2\), then
    \[
    \cL^1\bigl(D_x(\spt \mu)\bigr)>0
    \qquad\text{for \(\mu\)-almost all }x\in\Rd.
    \]

    \item Fix \(\theta\in(0,1]\). Then
    \[
    \hd D_x(\spt \mu)\ge \min\left\{\frac{\fs\mu-d\theta}{2},1\right\}
    \qquad\text{for \(\mu\)-almost all }x\in\Rd.
    \]
    In particular, if
    \[
    \fs\mu > 2+d\theta,
    \]
    then
    \[
    \cL^1\bigl(D_x(\spt \mu)\bigr)>0
    \qquad\text{for \(\mu\)-almost all }x\in\Rd.
    \]
\end{enumerate}
\end{theorem}

We prove Theorem \ref{thm-theta-zero} in Section \ref{sec:proof2}. Next we state a corollary which combines the large $\theta$ and small $\theta$ estimates from Theorems \ref{thm-ref} and \ref{thm-theta-zero}.  The aim is to bound the Hausdorff dimension of the pinned distance set of a single set from below for a large set of pins.  Again we start with a general result and then specialise to  statements which are easier to digest.

\begin{corollary}\label{cormain}
Let \(d\ge 2\) and let \(\mu\) be a compactly supported probability measure on \(\Rd\).  Fix \(\theta\in[0,1]\), and write
\[
s=\dim_{\mathrm{F}}^\theta \mu.
\]
Define
\[
\beta(s):=
\begin{cases}
\dfrac{s}{2}, & \theta=0,\\[2ex]
\max\!\left\{\dfrac{s-d\theta}{2},\,
\min\left\{\dfrac{s(1+2\theta)-d\theta}{2},\dfrac{s}{2},\dfrac{s^2}{d}\right\}\right\},
& 0<\theta<\dfrac12,\\[3ex]
\min\left\{s-\dfrac{d\theta}{2},\dfrac{s}{2}\right\},
& \dfrac12\le \theta\le 1.
\end{cases}
\]
If \(\beta(s)>0\), then \[
\hd D_x(\spt\mu)\ge \min\{1,\beta(s)\}
\qquad\text{for \(\mu\)-almost all }x.
\]

Moreover, if \(\beta(s)>1\), then \(\cL^1\bigl(D_x(\spt\mu)\bigr)>0\) for \(\mu\)-almost all \(x\).
\end{corollary}

\begin{remark}
In the regime \(0<\theta<\frac12\), the quantity
\[
\max\!\left\{\frac{s-d\theta}{2},\,
\min\left\{\frac{s(1+2\theta)-d\theta}{2},\frac{s}{2},\frac{s^2}{d}\right\}\right\}
\]
admits the equivalent piecewise form
\[\beta(s)=
\begin{cases}
\dfrac{s(1+2\theta)-d\theta}{2}, & 0\le s\le d\theta,\\[2ex]
\max\!\left\{\dfrac{s-d\theta}{2},\dfrac{s^2}{d}\right\},
& d\theta\le s\le \dfrac d2,\\[3ex]
\dfrac{s}{2}, & \dfrac d2\le s\le d.
\end{cases}
\]
\end{remark}

We now record certain special cases of the previous corollary corresponding to the values \(\theta=1,1/2,0\). Recall that \(\dim_{\mathrm{F}}^1\mu=\sd\mu\) and \(\dim_{\mathrm{F}}^0\mu=\fd\mu\).

\begin{corollary}\label{cor2}
Let \(d\ge 4\) and let \(\mu\) be a compactly supported probability measure on \(\Rd\).
\begin{enumerate}
    \item \(\hd D_x(\spt\mu)\ge \min\left\{\sd\mu-\frac d2,1\right\}\) for \(\mu\)-almost all \(x\).

    \item \(\hd D_x(\spt\mu)\ge \min\left\{\dim_{\mathrm{F}}^{1/2}\mu-\frac d4,1\right\}\) for \(\mu\)-almost all \(x\).

    \item \(\hd D_x(\spt\mu)\ge \min\left\{\frac{\fd\mu}{2},1\right\}\) for \(\mu\)-almost all \(x\).
\end{enumerate}
\end{corollary}
It is also of particular interest to determine conditions which guarantee that the distance set has full dimension. The next corollary gives simple, checkable, and relatively mild conditions depending on the Fourier spectrum.

\begin{corollary}\label{cor3}
Let \(d\ge 4\) and let \(\mu\) be a compactly supported probability measure on \(\Rd\). For $\theta \in [0,1]$, define
\[
T_d(\theta):=
\begin{cases}
2+d\theta, & 0 \le \theta < \dfrac{\sqrt d - 2}{d},\\[1ex]
\sqrt d, & \dfrac{\sqrt d - 2}{d} \le \theta < d^{-1/2},\\[1ex]
\dfrac{2+d\theta}{1+2\theta}, & d^{-1/2}\le \theta < \dfrac12,\\[2ex]
1+\dfrac{d\theta}{2}, & \dfrac12 \le \theta\le 1.
\end{cases}
\]
If, for some \(\theta\in[0,1]\),
\(
\dim_{\mathrm{F}}^\theta\mu \ge T_d(\theta),
\)
then
\[
\hd D(\spt\mu)=1.
\]
Moreover,
\(
\hd D_x(\spt\mu)=1 
\)
 for \(\mu\)-almost all $x$.
\end{corollary}

 \begin{remark}
For simplicity of exposition, we chose to present the results in this section in terms of the Fourier dimension and spectrum of measures rather than sets.  However, by appealing to the definition of the Fourier spectrum of sets, one may produce   statements about the distance set $D(E)$ of a given Borel set $E$  based on the Fourier spectrum of $E$.  We leave the details to the reader.  Furthermore, our results on pinned distance sets are achieved for `almost all pins', but these statements can  be turned into statements about sets of pins of large Hausdorff dimension.  For example, if $\hd D_x(\spt \mu) = 1$ for $\mu$-almost all $x$, then $\hd D_x(\spt \mu) = 1$ for a set of pins $x \in \spt \mu$ of Hausdorff dimension at least $\hd \mu$.  Moreover,   $\hd \mu$ can in turn be bounded from below by $\fs \mu$ for any $\theta \in [0,1]$ of interest. Again, we leave the details to the reader. 
 \end{remark}

\section{Sharpness examples} \label{sec:examples}

In this section we construct examples pertaining to the sharpness (or near sharpness) of our results from the previous section.  

First we build an example of a pair of compact sets which together witness very few distances but yet both satisfy good Fourier analytic estimates.

\begin{proposition}\label{ex1}
Let $d\ge 4$, and set
\[
k_1=\left\lceil \frac d2\right\rceil,
\qquad
k_2=\left\lfloor \frac d2\right\rfloor,
\qquad
k_1+k_2=d.
\]
Then, there exist compact sets $E,F\subset\R^d$ such that:
\begin{enumerate}[label=\textup{(\roman*)},itemsep=2pt]
\item
Only a single distance is realised between $E$ and $F$, that is,
\[
D(E,F)=\{\sqrt 2\},
\qquad\text{and so}\qquad
\hd D(E,F)=0;
\]
\item The Fourier spectra satisfy
\[
\fs E   \geq \min\{\theta k_1,  k_1-1\}  
\]
and
\[
\fs F   \geq \min\{\theta k_2,  k_2-1\}.
\]
\end{enumerate}
In particular, when $\theta=1/2$, $\dim_\mathrm{F}^{1/2} E + \dim_\mathrm{F}^{1/2} F = d/2$ and so the bound from Corollary \ref{cor1/2} is sharp.
\end{proposition}

\begin{proof}
Write $\R^d=\R^{k_1}\times\R^{k_2}$, and let
\[
E=S^{k_1-1}\times\{0\},
\qquad
F=\{0\}\times S^{k_2-1}.
\]
Let $\sigma_{k_1-1}$ and $\sigma_{k_2-1}$ denote the normalized surface measures on the corresponding spheres, and define
\[
\mu=\sigma_{k_1-1}\otimes\delta_0,
\qquad
\nu=\delta_0\otimes\sigma_{k_2-1}.
\]
Then $\mu$ and $\nu$ are probability measures supported on $E$ and $F$, respectively.

If $x=(x_1,0)\in E$ and $y=(0,y_2)\in F$, then
\[
|x-y|^2=|x_1|^2+|y_2|^2=1+1=2.
\]
Therefore, $D(E,F)=\{\sqrt 2\}$ and so $\hd D(E,F)=0$.

Since $\mu$ and $\nu$ are product measures, the desired estimates for the Fourier spectra follow from \cite[Theorem 2.1]{JMF2}  where we use  the fact that the surface measure on the sphere is Salem (and so has Fourier spectrum constantly equal to the Hausdorff dimension). This completes the proof.
\end{proof}

In the next example, we  allow larger Fourier spectra, while keeping the distance set with dimension strictly smaller than 1. 

\begin{proposition}\label{ex2}
    Let $k_1,k_2\ge 2$,  set
\[
d=k_1+k_2+1
\]
and fix $\alpha \in (0,\frac12)$. Then there exist compact sets $E,F\subset \R^d$  such that:
\begin{enumerate}[label=\textup{(\roman*)},itemsep=2pt]
\item The sets $E$ and $F$ realise a small set of distances, that is, 
\[
\hd D(E,F)\le 2\alpha<1;
\]
\item The Fourier spectra satisfy: 
\[
\fs E \geq \min\{(k_1+1)\theta, k_1 \theta+\alpha, k_1-1+\alpha\}
\]
and
\[
\fs F \geq \min\{(k_2+1)\theta, k_2 \theta+\alpha, k_2-1+\alpha\}.
\]
\end{enumerate}
\end{proposition}
\begin{proof}
    By \cite[Theorem~2]{Chen}, for the chosen $\alpha\in(0,1)$ there exists a compact set $A\subset[0,1]$ and a probability measure $\lambda$ supported on $A$ such that
\[
\lambda([x-r,x+r])\approx r^\alpha
\qquad (x\in A,\ 0<r<1),
\]
and
\[
|\what\lambda(\tau)|\lesssim |\tau|^{-\alpha/2}(\log |\tau|)^{1/2}
\qquad (|\tau|\ge 2).
\]
The lower regularity bound shows that $\supp\lambda=A$, and the two-sided interval estimate implies that $A$ is $\alpha$-Ahlfors regular. The Fourier decay bound implies that $\fd \lambda \geq \alpha$ and since it is supported on a set with Hausdorff dimension $\alpha$ we get that $\fs \lambda = \alpha$ for all $\theta \in [0,1]$. Now set
\[
A_1=1+A\subset[1,2],
\qquad
A_2=5+A\subset[5,6],
\]
and let $\lambda_1,\lambda_2$ be the translates of $\lambda$ to $A_1$ and $A_2$, and observe that translation preserves Fourier spectrum. Let $\sigma_{k_1-1}$ be the normalised surface measure on $S^{k_1-1}\subset\R^{k_1}$, and define
\[
\mu=\sigma_{k_1-1}\otimes\lambda_1\otimes\delta_0,
\qquad
\nu=\delta_0\otimes\lambda_2\otimes\sigma_{k_2-1}
\]
on $\R^{k_1}\times\R\times\R^{k_2} = \R^{k_1+k_2+1}=\R^d$. Thus
\[
E=S^{k_1-1}\times A_1\times\{0\},
\qquad
F=\{0\}\times A_2\times S^{k_2-1}.
\]

Take $x=(u,t_1,0)\in E$ and $y=(0,t_2,v)\in F$. Then
\[
|x-y|^2=|u|^2+|t_1-t_2|^2+|v|^2 = 2+|t_1-t_2|^2.
\]
Therefore
\[
D(E,F)=\left\{ \sqrt{2+|t_1-t_2|^2} : t_1\in A_1,\ t_2\in A_2\right\}.
\]
Since $A_2-A_1\subset[3,5]$, the map
\[
f(t)=\sqrt{2+t^2}
\]
is bi-Lipschitz on $[3,5]$. Hence
\[
\hd D(E,F)=\hd(A_2-A_1).
\]
Since $A_1$ and $A_2$ are $\alpha$-Ahlfors regular, we have $\hd(A_1-A_2)\le \hd (A_1 \times A_2) \leq 2\alpha$.

Since $\mu$ and $\nu$ are product measures, the desired lower bounds for the Fourier spectra of $\mu$ and $\nu$, which then pass to $E$ and $F$, once again follow from \cite[Theorem 2.1]{JMF2}, although this time we need to apply the estimates twice. This completes the proof.
\end{proof}

The following examples are immediate from the sharpness examples for the Hausdorff dimension version of the problem.  It is useful to record them and to keep them in mind here.  In particular, for $\theta=1/2$ we can never go lower than the threshold $d/4$ and so our threshold $d/4+1$ from Corollary \ref{cor3} is near optimal.  
\begin{proposition}\label{prop:64}
Let $d\ge 2$ and let $0<s<d/2$. Then, there exists $E \subseteq \rd$ such that $\hd D(E)<1$ and $\fs E \geq s \theta$ for all $\theta \in [0,1]$.
\end{proposition}

\begin{proof}
It follows from \cite{distance} (see also \cite[Example 4.8]{mattila}) that there exists a compact set $E\subset \R^d$ such that
\[
\hd E=s
\qquad\text{and}\qquad
\hd D(E)< 1.
\]
But then concavity of the Fourier spectrum (for measures) implies $\fs E \geq s \theta$ for all $\theta \in [0,1]$.
\end{proof}

Proposition \ref{prop:64} gives some simple but useful information regarding the optimal threshold for the Fourier spectrum to ensure a distance set with full Hausdorff dimension.  Recall the threshold we obtain in Corollary \ref{cor3}, and see the following section for a more detailed discussion.  We conclude this section with another simple observation which improves on the lower bound from Proposition \ref{prop:64} for small $\theta$.

\begin{proposition} \label{prop:easy1/2}
    Let $d\geq 2$ and let $0<s<1/2$.  Then there exists $E \subseteq \rd$ such that $\hd D(E)<1$ and $\fs E \geq \fd E \geq s$ for all $\theta \in [0,1]$.
\end{proposition}
\begin{proof}
    Let $E \subseteq \rd$ be such that $\ubd E = \hd E = \fd E = s$.  That is, we require $E$ to be a Salem set of dimension $s$ with the additional property that the upper box dimension is also given by $s$.  Such sets are well-known to exist and can be constructed, for example, by taking the image of a self-similar set in $[0,1]$ of (box and Hausdorff) dimension $s/2$ under $d$-dimensional Brownian motion $B: [0,\infty) \to \rd$. The image is almost surely a Salem set of dimension $s$ by \cite[Theorem 1, Chapter 17]{kahane} and, moreover, the box dimension cannot exceed $s$ since $d$-dimensional Brownian motion  is almost surely $\alpha$-Holder for all $0<\alpha<1/2$.  But then
    $D(E)$ is a Lipschitz image of the product set $E \times E$ under the map $(x,y) \mapsto |x-y|$ and so 
    \[
    \hd D(E) \leq \hd (E \times E) \leq \hd E + \ubd E = 2s<1,
    \]
    as required. We needed information about the box dimension because the inequality $\hd (E \times E) \leq \hd E + \hd E $ does not hold in general.
\end{proof}

\section{Conjectures}
In this section we consider  Corollary~\ref{cor3} which gives thresholds for the Fourier spectrum which ensure the distance set has full dimension.  By considering certain test points and heuristics, we formulate what we believe the sharp threshold should be.

Recall Corollary \ref{cor3} shows that the  distance set of the support of a measure $\mu$ has full Hausdorff dimension follows whenever 
\[
\fs \mu \ge T_d(\theta)
\]
for some \(\theta\in[0,1]\), where
\[
T_d(\theta):=
\begin{cases}
2+d\theta, & 0 \le \theta < \frac{\sqrt d - 2}{d},\\[1ex]
\sqrt d, & \frac{\sqrt d - 2}{d} \le \theta < d^{-1/2},\\[1ex]
\dfrac{2+d\theta}{1+2\theta}, & d^{-1/2}\le \theta\le \dfrac12,\\[2ex]
1+\dfrac{d\theta}{2}, & \dfrac12\le \theta\le 1.
\end{cases}
\]
This piecewise result arises naturally from our  method of proof. Nevertheless, we do not expect it to be optimal in all regimes. In particular, the flat branch
\[
T_d(\theta)=\sqrt d
\qquad
\left(\frac{\sqrt d - 2}{d} \le \theta < d^{-1/2}\right)
\]
appears to be an analytic byproduct of the present argument, more specifically, of the loss of \(\theta\)-dependence in the passage from \(L^2\)- to \(L^4\)-type estimates, rather than a reflection of the underlying geometry. To formulate a more realistic conjectural picture, it is natural to identify the expected values of the threshold at two anchor points, namely \(\theta=1\) and \(\theta=\frac12\). 

The first anchor point is the endpoint \(\theta=1\). By definition, $\dim_\mathrm{F}^1 \mu = \sd \mu$ is the Sobolev dimension of \(\mu\), so the condition \(\sd \mu >s\) is equivalent to the finiteness of the \(s\)-energy of \(\mu\). Thus, the endpoint \(\theta=1\) corresponds to the classical energy formulation of the distance set problem, with no additional gain coming from intermediate Fourier spectrum information. From this perspective, it is natural to conjecture that the optimal endpoint should agree with the Falconer threshold, i.e.
\[
T_d^{\mathrm{conj}}(1)=\frac d2.
\]
Our current method yields instead the threshold 
\[
T_d(1)=1+\frac d2,
\]
which is not expected to be optimal.

The second anchor point is the midpoint \(\theta=\frac12\), which is naturally tied to the \(L^4\)-type mechanism appearing in our argument. Based on the sharp restriction phenomena observed in the finite field setting in \cite{CGKPTZ25}, we believe that the correct threshold at this point should be
\[
T_d^{\mathrm{conj}}\!\left(\frac12\right)=\frac d4+\frac12,
\]
rather than the value \(\frac d4+1\) taken from Corollary~\ref{cor3}. 

If one assumes these two anchor values, then the simplest curve joining them is the affine line
\begin{equation}\label{eq:affine-threshold-model}
T_d^{\mathrm{conj}}(\theta)
=
1+\Bigl(\frac d2-1\Bigr)\theta,
\qquad 0\le \theta\le 1.
\end{equation}
This line has several interesting features. First, it recovers exactly the two anchor values:
\[
T_d^{\mathrm{conj}}\!\left(\frac12\right)=\frac d4+\frac12,
\qquad
T_d^{\mathrm{conj}}(1)=\frac d2.
\]
Second, it is strictly increasing on \([0,1]\), and therefore removes the flat  branch. In particular, at the transition point \(\theta=d^{-1/2}\), it gives
\begin{equation}\label{eq:Td-sqrtd-inverse}
T_d^{\mathrm{conj}}(d^{-1/2})
=
1+\Bigl(\frac d2-1\Bigr)\frac1{\sqrt d}
=
1+\frac{\sqrt d}{2}-\frac1{\sqrt d},
\end{equation}
which is substantially smaller than the currently proved value \(\sqrt d\). Third, at the  endpoint \(\theta=0\), we get
\[
T_d^{\mathrm{conj}}(0)=1.
\]
This suggests the possibility that, for Salem-type sets with essentially optimal Fourier decay, the true threshold for full distance set dimension may be 1 (independent of the ambient dimension).  Note that we   obtain the threshold  2 above which is already independent of the ambient dimension.  The threshold 1 is natural since the distance set is related to nonlinear projections and the Fourier dimension behaves very well with respect to linear orthogonal projections.  More precisely, the pinned distance map $y \mapsto |x-y|$ is a nonlinear projection and maps a set $E$ into the pinned distance set $D_x(E) \subseteq D(E)$ for every $x \in E$.  Moreover, for \emph{every} orthogonal projection $\pi$ of $\rd$ onto lines 
\[
\hd \pi(E) \geq \min\{\fd E,1\}.
\]
Therefore, one might expect $\fd E = 1$ to be enough to ensure that the distance set has full Hausdorff dimension.  That said, one has to be cautious with this `linearisation heuristic' because it is not true that all  pinned distance maps send sets with Fourier dimension at least 1 to sets with Hausdorff dimension 1.  For example, consider the sphere $S^{d-1} \subseteq \rd$ which has Fourier dimension $d-1$. However, the pinned distance map $y \mapsto |y|$ (that is, the pin is the point $x=0$) sends the sphere to a single point.

Given the above discussion, we are  led to the following conjecture.

\begin{conjecture}\label{conj:threshold-curve}
The optimal threshold ensuring the distance set dimension has full dimension is
\[
T_d^{\mathrm{conj}}(\theta)=1+\Bigl(\frac d2-1\Bigr)\theta,
\qquad 0\le \theta\le 1.
\]
That is, if $\mu$ is such that  $\fs \mu \geq T_d^{\mathrm{conj}}(\theta)$ for some $\theta \in [0,1]$, then $\hd D(\spt \mu) = 1.$
\end{conjecture}

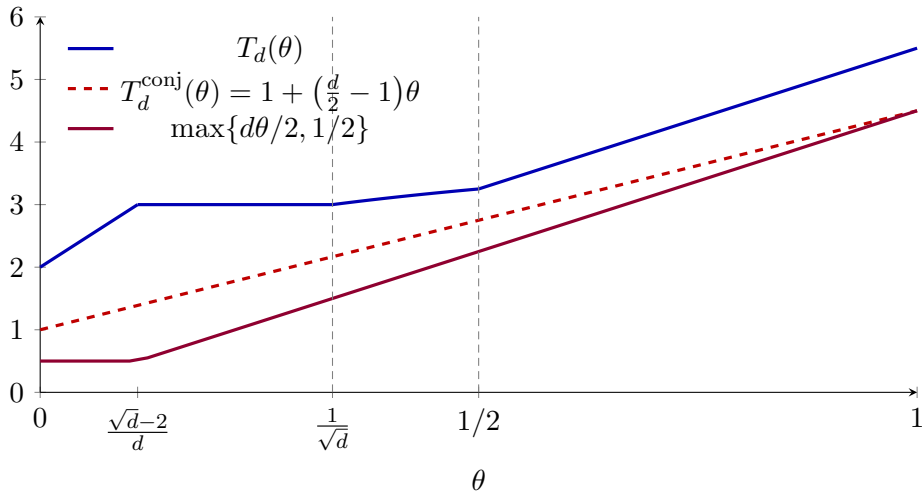
\begin{figure}[ht]
\centering

\def\dd{9} 

\begin{tikzpicture}
\begin{axis}[
    width=0.8\textwidth,
    height=0.4\textwidth,
    xmin=0, xmax=1,
    ymin=0, ymax={1+\dd/2+0.5},
    axis lines=left,
    xlabel={$\theta$},
    ylabel={},
    samples=50,
    no markers,
    clip=false,
    xtick={0,{(sqrt(\dd)-2)/\dd},{1/sqrt(\dd)},0.5,1},
    xticklabels={$0$,$\frac{\sqrt{d}-2}{d}$,$\frac{1}{\sqrt{d}}$,$1/2$,$1$},
    ytick={0,1,2,3,4,5,6},
    legend style={
        at={(0.02,0.98)},
        anchor=north west,
        draw=none,
        fill=none
    }
]


\addplot[very thick, blue!70!black, domain=0:{(sqrt(\dd)-2)/\dd}, forget plot] {2+\dd*x};
\addplot[very thick, blue!70!black, domain={(sqrt(\dd)-2)/\dd}:{1/sqrt(\dd)}, forget plot] {sqrt(\dd)};
\addplot[very thick, blue!70!black, domain={1/sqrt(\dd)}:0.5, forget plot] {(2+\dd*x)/(1+2*x)};
\addplot[very thick, blue!70!black, domain=0.5:1, forget plot] {1+\dd*x/2};

\addplot[very thick, red!75!black, dashed, domain=0:1, forget plot] {1+(\dd/2-1)*x};

\addplot[very thick, purple!75!black, domain=0:1, forget plot] {max(\dd*x/2,1/2)};

\addplot[densely dashed, gray, forget plot]
coordinates {({1/sqrt(\dd)},0) ({1/sqrt(\dd)},{1+\dd/2+0.5})};
\addplot[densely dashed, gray, forget plot]
coordinates {(0.5,0) (0.5,{1+\dd/2+0.5})};

\addlegendimage{very thick, blue!70!black}
\addlegendentry{$T_d(\theta)$}

\addlegendimage{very thick, red!75!black, dashed}
\addlegendentry{$T_d^{\mathrm{conj}}(\theta)=1+\bigl(\frac d2-1\bigr)\theta$}

\addlegendimage{very thick, purple!75!black}
\addlegendentry{$\max\{d\theta/2,1/2\}$}

\end{axis}
\end{tikzpicture}

\caption{Comparison between the currently proved threshold $T_d(\theta)$ and the conjectural affine line $T_d^{\mathrm{conj}}(\theta)$ with $d=9$. We also include the known lower bound for the threshold coming from Propositions \ref{prop:64} and \ref{prop:easy1/2}.}
\label{fig:threshold-curve-conjecture}
\end{figure}

\begin{figure}[ht]
\centering

\def\dd{40} 

\begin{tikzpicture}
\begin{axis}[
    width=0.8\textwidth,
    height=0.5\textwidth,
    xmin=0, xmax=1,
    ymin=0, ymax={1+\dd/2+0.5},
    axis lines=left,
    xlabel={$\theta$},
    ylabel={},
    samples=50,
    no markers,
    clip=false,
    xtick={0,{(sqrt(\dd)-2)/\dd},{1/sqrt(\dd)},0.5,1},
    xticklabels={
        $0$,
        $\hspace*{-0.8em}\frac{\sqrt{d}-2}{d}$,
        $\hspace*{0.9em}\frac{1}{\sqrt{d}}$,
        $1/2$,
        $1$
    },
    ytick={0,5,10,15,20},
    legend style={
        at={(0.02,0.98)},
        anchor=north west,
        draw=none,
        fill=none
    }
]


\addplot[very thick, blue!70!black, domain=0:{(sqrt(\dd)-2)/\dd}, forget plot] {2+\dd*x};
\addplot[very thick, blue!70!black, domain={(sqrt(\dd)-2)/\dd}:{1/sqrt(\dd)}, forget plot] {sqrt(\dd)};
\addplot[very thick, blue!70!black, domain={1/sqrt(\dd)}:0.5, forget plot] {(2+\dd*x)/(1+2*x)};
\addplot[very thick, blue!70!black, domain=0.5:1, forget plot] {1+\dd*x/2};

\addplot[very thick, red!75!black, dashed, domain=0:1, forget plot] {1+(\dd/2-1)*x};

\addplot[very thick, purple!75!black, domain=0:1, forget plot] {max(\dd*x/2,1/2)};

\addplot[densely dashed, gray, forget plot]
coordinates {({1/sqrt(\dd)},0) ({1/sqrt(\dd)},{1+\dd/2+0.5})};
\addplot[densely dashed, gray, forget plot]
coordinates {(0.5,0) (0.5,{1+\dd/2+0.5})};

\addlegendimage{very thick, blue!70!black}
\addlegendentry{$T_d(\theta)$}

\addlegendimage{very thick, red!75!black, dashed}
\addlegendentry{$T_d^{\mathrm{conj}}(\theta)=1+\bigl(\frac d2-1\bigr)\theta$}

\addlegendimage{very thick, purple!75!black}
\addlegendentry{$\max\{d\theta/2,1/2\}$}

\end{axis}
\end{tikzpicture}

\caption{Comparison between the currently proved threshold $T_d(\theta)$ and the conjectural affine line $T_d^{\mathrm{conj}}(\theta)$ with $d=40$. We also include the known lower bound for the threshold coming from Propositions \ref{prop:64} and \ref{prop:easy1/2}.}
\label{fig:threshold-curve-conjecture1}
\end{figure}
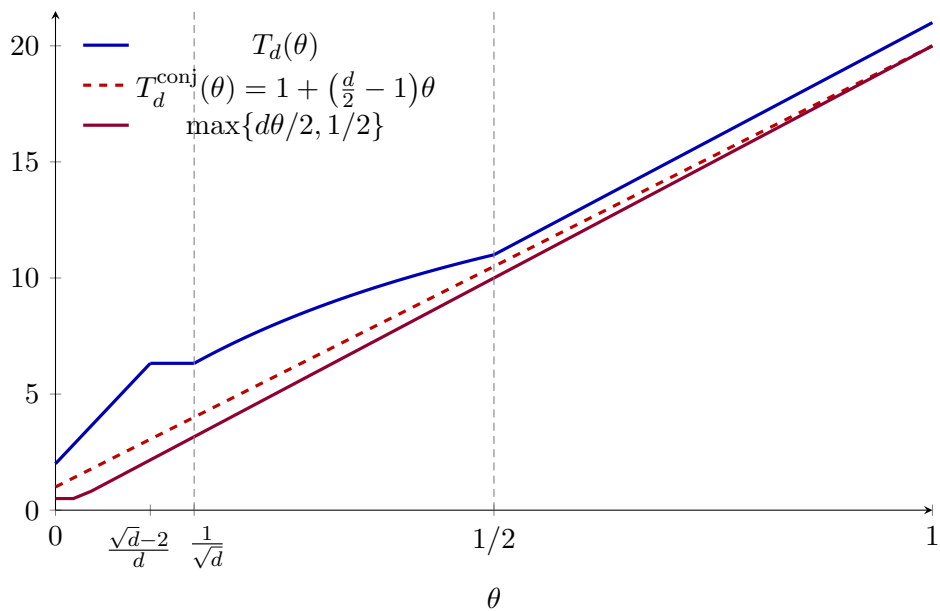

\section{Proof of Theorem \ref{thm-ref}} \label{sec:proof}

In this section we prove our main result, Theorem \ref{thm-ref}.  Without loss of generality, we assume that $\theta_1, \theta_2 \in (0,1]$.  We can do this because the desired bounds are continuous in $\theta_1, \theta_2$. Moreover, we assume without loss of generality that  $0<u \leq d$.  Indeed, if $u=0$ then the bounds are trivial and there is nothing to prove.    We start by introducing some notation which will be used throughout. 
\\
\paragraph{\bf Notation}

Although our main results pertain to standard distance sets, it will be convenient to work with squared distances.
For Borel sets $E,F\subseteq\Rd$, write
\[
D(E,F)=\{|x-y|:x\in E,\ y\in F\}
\qquad\text{and}\qquad
D^2(E,F)=\{|x-y|^2:x\in E,\ y\in F\}.
\]
In particular, $D(E)=D(E,E)$ and $D^2(E)=D^2(E,E)$. It is easy to see that
\[
\hd D^2(E,F)=\hd D(E,F).
\]
We write $X\lesssim_{\gamma}Y$ to mean $X\le C_{\gamma}Y$ for a constant $C_{\gamma}$
depending only on the indicated parameters $\gamma$ (and $d$ if not listed), and write $X\approx_{\gamma} Y$
when both $X\lesssim_{\gamma}Y$ and $Y\lesssim_{\gamma}X$ hold.  We set $\ang{x}=(1+|x|^2)^{1/2}$.

Fix compactly supported probability measures $\mu_1,\mu_2$ on $\Rd$.  Choose $R\ge 1$ such that
\[
\supp\mu_1\cup\supp\mu_2\subset B(0,R).
\]
For $x\in\Rd$, define the \emph{pinned squared-distance measure} by
\[
\delta_{x,\mu_2}^2(B):=\mu_2\bigl(\{y\in\Rd:|x-y|^2\in B\}\bigr),
\]
for Borel sets $B\subset\R$. Thus, $\delta_{x,\mu_2}^2$ is the pushforward of $\mu_2$ under the map
$y\mapsto |x-y|^2$, and
\[
\supp\delta_{x,\mu_2}^2=D_x(\supp\mu_2)^2 = D^2(\{x\}, \spt \mu_2).
\]
Define the map
\begin{equation}
\label{eq:lift-map}
F:\Rd\to\Rdpo,\qquad \text{by} \qquad  F(y)=(y,|y|^2),
\end{equation}
and the lifted measure
\begin{equation}
\label{eq:nu-def}
\nu\coloneqq F_{\#}\mu_2\quad\text{on }\Rdpo.
\end{equation}
The following lemma expresses the Fourier transform of $\delta_{x,\mu_2}^2$
in terms of $\widehat{\nu}$ evaluated along the line through the origin in
the direction $(-2x,1)$, and then derives Hausdorff dimension and
positive Lebesgue measure criteria for the pinned distance set
$D_x(\supp\mu_2)$.
\begin{lemma}\label{lem:pinned-distance-measure-fourier}
Let $x\in\Rd$. Then, for every $\tau\in\R$,
\begin{equation}\label{eq:pinned-fourier}
\widehat{\delta_{x,\mu_2}^2}(\tau)
=
e^{-2\pi i\tau|x|^2}\,\widehat\nu\bigl(\tau(-2x,1)\bigr),
\end{equation}
and, in particular,
\begin{equation}\label{eq:pinned-fourier-abs}
\bigl|\widehat{\delta_{x,\mu_2}^2}(\tau)\bigr|^2
=
\bigl|\widehat\nu\bigl(\tau(-2x,1)\bigr)\bigr|^2.
\end{equation}
Consequently, if
\[
\int_{\R}\bigl|\widehat{\delta_{x,\mu_2}^2}(\tau)\bigr|^2|\tau|^{\sigma-1}\,d\tau<\infty
\]
for some $0<\sigma<1$, then
\[
\hd D_x(\supp\mu_2)\ge \sigma.
\]
Moreover, if
\[
\int_{\R}\bigl|\widehat{\delta_{x,\mu_2}^2}(\tau)\bigr|^2\,d\tau<\infty,
\]
then
\[
\cL^1\bigl(D_x(\supp\mu_2)\bigr)>0.
\]
\end{lemma}

\begin{proof}
By definition,
\[
\widehat{\delta_{x,\mu_2}^2}(\tau)
=
\int_{\Rd}e^{-2\pi i\tau|x-y|^2}\,d\mu_2(y).
\]
Expanding the square gives
\begin{align*}
\widehat{\delta_{x,\mu_2}^2}(\tau)
&=
\int_{\Rd}e^{-2\pi i\tau(|x|^2-2x\cdot y+|y|^2)}\,d\mu_2(y) \\
&=
 e^{-2\pi i\tau|x|^2}
\int_{\Rd}e^{-2\pi i\tau(-2x\cdot y+|y|^2)}\,d\mu_2(y) \\
&=
 e^{-2\pi i\tau|x|^2}
\int_{\Rdpo}e^{-2\pi i\tau(-2x,1)\cdot w}\,d\nu(w)
=
 e^{-2\pi i\tau|x|^2}\,\widehat\nu\bigl(\tau(-2x,1)\bigr).
\end{align*}
This proves \eqref{eq:pinned-fourier}, and \eqref{eq:pinned-fourier-abs} is immediate.

If
\[
\int_{\R}\bigl|\widehat{\delta_{x,\mu_2}^2}(\tau)\bigr|^2|\tau|^{\sigma-1}\,d\tau<\infty
\]
for some $0<\sigma<1$, then by \eqref{energyequiv}
\[
I_{\sigma}(\delta_{x,\mu_2}^2)<\infty.
\]
Hence
\[
\hd D_x(\supp\mu_2)
=
\hd D_x(\supp\mu_2)^2
=
\hd\supp\delta_{x,\mu_2}^2
\ge \sigma.
\]

If instead
\[
\int_{\R}\bigl|\widehat{\delta_{x,\mu_2}^2}(\tau)\bigr|^2\,d\tau<\infty,
\]
then Plancherel implies that $\delta_{x,\mu_2}^2$ is absolutely continuous with an $L^2$ density. Since $\delta_{x,\mu_2}^2$ has total mass $1$, that density is non-zero, and therefore
\[
\cL^1\bigl(D_x(\supp\mu_2)^2\bigr)=\cL^1\bigl(\supp\delta_{x,\mu_2}^2\bigr)>0.
\]
Therefore
\[
\cL^1\bigl(D_x(\supp\mu_2)\bigr)>0.
\] 
This completes the proof.
\end{proof}

Define
\begin{equation}
\label{eq:A-def}
A(\tau)\coloneqq \int_{\Rd} \bigl|\widehat\nu\bigl(\tau(-2x,1)\bigr)\bigr|^2\,d\mu_1(x).
\end{equation}

Fix $\chi\in C_c^\infty(\Rd)$ with $\chi\equiv 1$ on $B(0,R)$.
Set
\[
G_{\tau}(x):=\bigl|\widehat\nu\bigl(\tau(-2x,1)\bigr)\bigr|^2,
\qquad \text{and} \qquad 
f_{\tau}(x):=\chi(x)G_{\tau}(x).
\]
Then $A(\tau)=\int f_{\tau}(x)\,d\mu_1(x)$.

The next lemma gives an explicit formula for \(\widehat f_\tau\).

\begin{lemma}
The following identity holds
\begin{equation}\label{eq:fhat-theta1-theta2}
\widehat f_{\tau}(\xi)
=
\iint e^{-2\pi i\tau(|y|^2-|z|^2)}\,\widehat\chi\bigl(\xi-2\tau(y-z)\bigr)
\,d\mu_2(y)\,d\mu_2(z).
\end{equation}
\end{lemma}

    \begin{proof}
Fix \(\tau\in\mathbb R\). Since \(\nu=F_{\#}\mu_2\) and \(F(y)=(y,|y|^2)\), we have for every
\(x\in\mathbb R^d\),
\begin{align*}
\widehat\nu\bigl(\tau(-2x,1)\bigr)
&=
\int_{\mathbb R^{d+1}} e^{-2\pi i\, \tau(-2x,1)\cdot w}\,d\nu(w)=
\int_{\mathbb R^d} e^{-2\pi i\, \tau(-2x,1)\cdot F(y)}\,d\mu_2(y) \\
&=
\int_{\mathbb R^d} e^{-2\pi i\, \tau(-2x\cdot y+|y|^2)}\,d\mu_2(y)=
\int_{\mathbb R^d} e^{-2\pi i\tau(|y|^2-2x\cdot y)}\,d\mu_2(y).
\end{align*}
Therefore,
\begin{align*}
G_\tau(x)
&=
\left|\widehat\nu\bigl(\tau(-2x,1)\bigr)\right|^2 \\
&=
\left(\int_{\mathbb R^d} e^{-2\pi i\tau(|y|^2-2x\cdot y)}\,d\mu_2(y)\right)
\left(\int_{\mathbb R^d} e^{\,2\pi i\tau(|z|^2-2x\cdot z)}\,d\mu_2(z)\right) \\
&=
\iint_{\mathbb R^d\times\mathbb R^d}
e^{-2\pi i\tau(|y|^2-|z|^2)}\,e^{4\pi i\tau x\cdot (y-z)}
\,d\mu_2(y)\,d\mu_2(z).
\end{align*}

Since \(\chi\in C_c^\infty(\mathbb R^d)\subset L^1(\mathbb R^d)\) and \(\mu_2\) is finite,
\[
\iiint_{\mathbb R^d\times\mathbb R^d\times\mathbb R^d}
|\chi(x)|\,dx\,d\mu_2(y)\,d\mu_2(z)
=
\|\chi\|_{L^1}\,\mu_2(\mathbb R^d)^2<\infty.
\]
Hence, Fubini's theorem applies, and we obtain
\begin{align*}
\widehat f_\tau(\xi)
&=
\int_{\mathbb R^d} e^{-2\pi i x\cdot \xi}\,\chi(x)G_\tau(x)\,dx \\
&=
\iiint
e^{-2\pi i x\cdot \xi}\chi(x)\,
e^{-2\pi i\tau(|y|^2-|z|^2)}\,e^{4\pi i\tau x\cdot (y-z)}
\,dx\,d\mu_2(y)\,d\mu_2(z) \\
&=
\iint
e^{-2\pi i\tau(|y|^2-|z|^2)}
\left(
\int_{\mathbb R^d}
\chi(x)e^{-2\pi i x\cdot(\xi-2\tau(y-z))}\,dx
\right)
\,d\mu_2(y)\,d\mu_2(z).
\end{align*}
The inner integral is exactly
$
\widehat\chi\bigl(\xi-2\tau(y-z)\bigr),
$
so
\[
\widehat f_{\tau}(\xi)
=
\iint e^{-2\pi i\tau(|y|^2-|z|^2)}\,\widehat\chi\bigl(\xi-2\tau(y-z)\bigr)
\,d\mu_2(y)\,d\mu_2(z).
\]
This proves \eqref{eq:fhat-theta1-theta2}.
\end{proof}
To bound \(A(\tau)=\int f_\tau\,d\mu_1\), we will use the following duality estimate, which converts \(\J_{u,\theta}(\mu_1)\)-control into a weighted Fourier norm bound for \(f_\tau\).
\begin{lemma}\label{lem:duality-general}
Let $0<\theta\le 1$ and $0<u<d\theta$. Assume $\J_{u,\theta}(\mu)<\infty$, where
$\mu$ is a compactly supported probability measure on $\Rd$. Then, for every $f\in C_c(\Rd)$ satisfying
\[
\int_{\Rd}|\widehat f(\xi)|^{\frac{2}{2-\theta}}|\xi|^{\frac{\theta d-u}{2-\theta}}\,d\xi<\infty,
\]
one has
\begin{equation}\label{eq:duality-general}
\left|\int f\,d\mu\right|
\le
\J_{u,\theta}(\mu)^{1/2}
\left(\int_{\Rd}|\widehat f(\xi)|^{\frac{2}{2-\theta}}|\xi|^{\frac{\theta d-u}{2-\theta}}\,d\xi\right)^{\frac{2-\theta}{2}}.
\end{equation}
\end{lemma}
\begin{proof}
Set
\[
p:=\frac{2}{\theta},
\qquad \text{and} \qquad 
q:=\frac{2}{2-\theta},
\]
so that \(2 \leq p<\infty\), \(1<q\le 2\), and
\[
\frac1p+\frac1q=1.
\]

We first prove \eqref{eq:duality-general} for Schwartz functions.  To this end, let \(f\in \mathcal S(\mathbb R^d)\). Since \(\widehat f\in L^1(\mathbb R^d)\) and
\(|\widehat\mu(\xi)|\le \mu(\mathbb R^d)\) for all \(\xi\), Fourier inversion and Fubini give
\[
\int_{\mathbb R^d} f(x)\,d\mu(x)
=
\int_{\mathbb R^d}\widehat f(\xi)\,\overline{\widehat\mu(\xi)}\,d\xi.
\]
Hence
\begin{align*}
\left|\int f\,d\mu\right|
&\le
\int_{\mathbb R^d} |\widehat f(\xi)|\,|\widehat\mu(\xi)|\,d\xi \\
&=
\int_{\mathbb R^d}
\Bigl(|\widehat\mu(\xi)|^{\frac2\theta}|\xi|^{\frac{u}{\theta}-d}\Bigr)^{\frac{\theta}{2}}
\Bigl(|\widehat f(\xi)|^{\frac{2}{2-\theta}}|\xi|^{\frac{\theta d-u}{2-\theta}}\Bigr)^{\frac{2-\theta}{2}}
\,d\xi.
\end{align*}
Applying Hölder's inequality with exponents \(p=2/\theta\) and \(q=2/(2-\theta)\), we obtain
\begin{align*}
\left|\int f\,d\mu\right|
&\le
\left(
\int_{\mathbb R^d}
|\widehat\mu(\xi)|^{\frac2\theta}|\xi|^{\frac{u}{\theta}-d}\,d\xi
\right)^{\frac{\theta}{2}}
\left(
\int_{\mathbb R^d}
|\widehat f(\xi)|^{\frac{2}{2-\theta}}|\xi|^{\frac{\theta d-u}{2-\theta}}\,d\xi
\right)^{\frac{2-\theta}{2}} \\
&=
\J_{u,\theta}(\mu)^{1/2}
\left(
\int_{\mathbb R^d}
|\widehat f(\xi)|^{\frac{2}{2-\theta}}|\xi|^{\frac{\theta d-u}{2-\theta}}\,d\xi
\right)^{\frac{2-\theta}{2}}.
\end{align*}
Thus, \eqref{eq:duality-general} holds for all \(f\in \mathcal S(\mathbb R^d)\).  We now pass to the case \(f\in C_c(\mathbb R^d)\) satisfying
\[
\int_{\mathbb R^d}
|\widehat f(\xi)|^{\frac{2}{2-\theta}}|\xi|^{\frac{\theta d-u}{2-\theta}}\,d\xi
<\infty.
\]
Choose \(\psi\in \mathcal S(\mathbb R^d)\) such that
\[
\widehat\psi\in C_c^\infty(\mathbb R^d),\qquad
0\le \widehat\psi\le 1,\qquad
\widehat\psi(\xi)=1 \ \text{for } |\xi|\le 1.
\]
For \(N\ge 1\), define
\[
\psi_N(x):=N^d\psi(Nx).
\]
Then, by the scaling property of the Fourier transform,
\[
\widehat{\psi_N}(\xi)=\widehat\psi(\xi/N).
\]
Now define \(f_N\) by
\[
\widehat{f_N}(\xi):=\widehat f(\xi)\widehat{\psi_N}(\xi).
\]
Since \(f\in C_c(\mathbb R^d)\), its Fourier transform \(\widehat f\) is \(C^\infty\), and since
\(\widehat{\psi_N}\in C_c^\infty(\mathbb R^d)\), it follows that
\(\widehat{f_N}\in C_c^\infty(\mathbb R^d)\). Hence \(f_N\in \mathcal S(\mathbb R^d)\). Therefore,
by the estimate already proved for Schwartz functions,
\begin{equation}\label{eq:duality-general-fN}
\left|\int f_N\,d\mu\right|
\le
\J_{u,\theta}(\mu)^{1/2}
\left(
\int_{\mathbb R^d}
|\widehat f(\xi)|^{\frac{2}{2-\theta}}
|\widehat{\psi_N}(\xi)|^{\frac{2}{2-\theta}}
|\xi|^{\frac{\theta d-u}{2-\theta}}\,d\xi
\right)^{\frac{2-\theta}{2}}.
\end{equation}

Since \(0\le \widehat{\psi_N}\le 1\) and
\[
\widehat{\psi_N}(\xi)=\widehat\psi(\xi/N)\to 1
\]
for every $\xi\in\mathbb R^d$, the dominated convergence theorem yields
\[
\int_{\mathbb R^d}
|\widehat f(\xi)|^{\frac{2}{2-\theta}}
|\widehat{\psi_N}(\xi)|^{\frac{2}{2-\theta}}
|\xi|^{\frac{\theta d-u}{2-\theta}}\,d\xi
\to
\int_{\mathbb R^d}
|\widehat f(\xi)|^{\frac{2}{2-\theta}}
|\xi|^{\frac{\theta d-u}{2-\theta}}\,d\xi.
\]

It remains to show that
\[
\int f_N\,d\mu\to \int f\,d\mu.
\]
Since $
\widehat{f_N}=\widehat f\,\widehat{\psi_N},
$
we have
$
f_N=f*\psi_N.
$

We next show that \(f_N\to f\) uniformly on \(\mathbb R^d\). First,
\[
\int_{\mathbb R^d}\psi_N(x)\,dx=\widehat{\psi_N}(0)=\widehat\psi(0)=1.
\]
Also, for every \(\delta>0\),
\[
\int_{|x|\ge \delta} |\psi_N(x)|\,dx
=
\int_{|y|\ge N\delta} |\psi(y)|\,dy
\to 0
\qquad (N\to\infty),
\]
because \(\psi\in L^1(\mathbb R^d)\). Thus, \((\psi_N)\) is an approximate identity in \(L^1\). A direct computation shows that
\[
\|f_N-f\|_{L^\infty(\mathbb R^d)}\to 0.
\]

Since \(\mu\) is finite, it follows that
\[
\left|\int (f_N-f)\,d\mu\right|
\le
\mu(\mathbb R^d)\,\|f_N-f\|_{L^\infty(\mathbb R^d)}
\to 0.
\]
Thus
\[
\int f_N\,d\mu\to \int f\,d\mu.
\]

We may now pass to the limit in \eqref{eq:duality-general-fN} and obtain
\[
\left|\int f\,d\mu\right|
\le
\J_{u,\theta}(\mu)^{1/2}
\left(
\int_{\mathbb R^d}
|\widehat f(\xi)|^{\frac{2}{2-\theta}}
|\xi|^{\frac{\theta d-u}{2-\theta}}\,d\xi
\right)^{\frac{2-\theta}{2}}.
\]
This is exactly \eqref{eq:duality-general}.
\end{proof}

We next provide a weighted $L^2$ estimate for $\widehat f_\tau$.

\begin{lemma}\label{lem:agenera}
Let $0<m<d$ and $0<\gamma<d/2$, and set
\[
\eta_2:=\mu_2*\widetilde\mu_2,
\qquad \text{and} \qquad 
\widetilde\mu_2(E):=\mu_2(-E).
\]
Then, for every $\tau\in\R$,
\begin{equation}\label{eq:L2-weighted-general}
\int_{\Rd}|\widehat f_{\tau}(\xi)|^2|\xi|^{d-m}\,d\xi
\lesssim_{d,m,\gamma,R}
(1+|\tau|)^{d-m-2\gamma}I_{2\gamma}(\eta_2).
\end{equation}
\end{lemma}
\begin{proof}
Fix $N>2\gamma$. Since $\widehat\chi\in\mathcal S(\Rd)$, all integrals below are
absolutely convergent, and hence Fubini's theorem applies. Using
\eqref{eq:fhat-theta1-theta2}, we obtain
\begin{align*}
\int_{\Rd}|\widehat f_{\tau}(\xi)|^2|\xi|^{d-m}\,d\xi
&\le
\iiiint
\Bigl|
\int_{\Rd}
|\xi|^{d-m}\widehat\chi(\xi-2\tau u)
\overline{\widehat\chi(\xi-2\tau u')}
\,d\xi
\Bigr|
\,d\mu_2(y)\,d\mu_2(z)\,d\mu_2(y')\,d\mu_2(z'),
\end{align*}
where $u=y-z$ and $u'=y'-z'$. After the change of variables
$\xi\mapsto \xi+2\tau u$, the inner integral becomes
\[
\int_{\Rd} |\xi+2\tau u|^{d-m}\widehat\chi(\xi)
\overline{\widehat\chi(\xi+r)}\,d\xi,
\qquad r:=2\tau(u-u').
\]
Since $0<m<d$, we have $d-m>0$, and therefore
\[
|\xi+2\tau u|^{d-m}
\le
\ang{\xi+2\tau u}^{\,d-m}
\lesssim
\ang{\tau u}^{\,d-m}\ang{\xi}^{\,d-m}
\]
by Peetre's inequality. Hence
\[
\Bigl|
\int_{\Rd} |\xi+2\tau u|^{d-m}\widehat\chi(\xi)
\overline{\widehat\chi(\xi+r)}\,d\xi
\Bigr|
\lesssim
\ang{\tau u}^{\,d-m}\,\kappa(r),
\]
where
\[
\kappa(r):=
\int_{\Rd}\ang{\xi}^{\,d-m}|\widehat\chi(\xi)|\,|\widehat\chi(\xi+r)|\,d\xi.
\]
Since $\widehat\chi\in\mathcal S(\Rd)$, for every $N>0$ one has
\[
\kappa(r)\lesssim_N \ang{r}^{-N}.
\]
Also, $\spt\mu_2\subset B(0,R)$ implies $|u|\le 2R$, and therefore
\[
\ang{\tau u}^{\,d-m}\lesssim_R (1+|\tau|)^{d-m}.
\]
Consequently,
\[
\Bigl|
\int_{\Rd}
|\xi|^{d-m}\widehat\chi(\xi-2\tau u)
\overline{\widehat\chi(\xi-2\tau u')}\,d\xi
\Bigr|
\lesssim_{N,R}
(1+|\tau|)^{d-m}(1+|\tau||u-u'|)^{-N}.
\]
It follows that
\begin{equation}\label{eq:kernel-bound-agenera}
\int_{\Rd}|\widehat f_{\tau}(\xi)|^2|\xi|^{d-m}\,d\xi
\lesssim_{N,R}
(1+|\tau|)^{d-m}
\iiiint (1+|\tau||u-u'|)^{-N}\,d\mu_2^{\otimes 4}.
\end{equation}

Since $u=y-z$ and $u'=y'-z'$ are distributed according to
$\eta_2=\mu_2*\widetilde\mu_2$, \eqref{eq:kernel-bound-agenera} becomes
\[
\int_{\Rd}|\widehat f_{\tau}(\xi)|^2|\xi|^{d-m}\,d\xi
\lesssim_{N,R}
(1+|\tau|)^{d-m}
\iint (1+|\tau||u-u'|)^{-N}\,d\eta_2(u)\,d\eta_2(u').
\]
Set
\[
J(\tau):=
\iint (1+|\tau||u-u'|)^{-N}\,d\eta_2(u)\,d\eta_2(u').
\]
We claim that
\begin{equation}\label{eq:kernel-to-energy-agenera}
J(\tau)\lesssim_{\gamma,N,R}(1+|\tau|)^{-2\gamma}I_{2\gamma}(\eta_2)
\qquad (\tau\in\R).
\end{equation}

If $|\tau|\ge 1$, then
\[
(1+|\tau|r)^{-N}\lesssim \sum_{k\ge 0}2^{-kN}\mathbf 1_{\{r\le 2^{k+1}/|\tau|\}},
\]
so
\[
J(\tau)\lesssim
\sum_{k\ge 0}2^{-kN}\,
\eta_2\otimes\eta_2\Bigl(\{|u-u'|\le 2^{k+1}/|\tau|\}\Bigr).
\]
For every $r>0$,
\[
I_{2\gamma}(\eta_2)
\ge
\iint_{|u-u'|\le r}|u-u'|^{-2\gamma}\,d\eta_2(u)\,d\eta_2(u')
\ge
r^{-2\gamma}\,\eta_2\otimes\eta_2\Bigl(\{|u-u'|\le r\}\Bigr),
\]
and hence
\[
\eta_2\otimes\eta_2\Bigl(\{|u-u'|\le r\}\Bigr)
\le r^{2\gamma}I_{2\gamma}(\eta_2).
\]
With $r=2^{k+1}/|\tau|$, this gives
\[
J(\tau)\lesssim
|\tau|^{-2\gamma}I_{2\gamma}(\eta_2)
\sum_{k\ge 0}2^{-kN}2^{2\gamma(k+1)}
\lesssim_{\gamma,N}
|\tau|^{-2\gamma}I_{2\gamma}(\eta_2),
\]
since $N>2\gamma$.

If $|\tau|\le 1$, then $J(\tau)\le \eta_2(\Rd)^2$. Also
$\spt\eta_2\subset B(0,2R)$, and hence $|u-u'|\le 4R$ on
$\spt\eta_2\times\spt\eta_2$. Therefore
\[
I_{2\gamma}(\eta_2)
=
\iint |u-u'|^{-2\gamma}\,d\eta_2(u)\,d\eta_2(u')
\ge
(4R)^{-2\gamma}\eta_2(\Rd)^2.
\]
Thus, $\eta_2(\Rd)^2\lesssim_{\gamma,R} I_{2\gamma}(\eta_2)$, and since
$(1+|\tau|)^{-2\gamma}\approx 1$ for $|\tau|\le 1$, we obtain
\eqref{eq:kernel-to-energy-agenera} in this range as well.

In other words, we get
\[
\int_{\Rd}|\widehat f_{\tau}(\xi)|^2|\xi|^{d-m}\,d\xi
\lesssim_{d,m,\gamma,R}
(1+|\tau|)^{d-m-2\gamma}I_{2\gamma}(\eta_2),
\]
which is exactly \eqref{eq:L2-weighted-general}.
\end{proof}

In the next step, we control $I_{2\gamma}(\eta_2)$ directly from $\J_{v,\theta_2}(\mu_2)$.

\begin{lemma}\label{lem:eta-energy-general-theta}
Let $0<\theta_2\le 1$ and $v>0$.
Set
\[
a(v):=\min\!\left\{\frac d2,\, v, \frac{v}{2\theta_2}\right\}.
\]
If $\J_{v,\theta_2}(\mu_2)<\infty$, then for every $0<\gamma<a(v)$,
\begin{equation}\label{eq:eta-energy-general-theta}
I_{2\gamma}(\eta_2)^{1/2}
\lesssim_{d,\gamma,\theta_2,v,R}
\J_{v,\theta_2}(\mu_2)^{\rho(\theta_2)},
\end{equation}
where \[
\rho(\theta_2):=
\begin{cases}
1, & 0<\theta_2\le \frac12,\\[1ex]
\dfrac{1}{2\theta_2}, & \frac12\le \theta_2\le 1.
\end{cases}
\]
\end{lemma}

\begin{proof}
Since $0<\gamma<d/2$, \eqref{energyequiv} gives
\[
I_{2\gamma}(\eta_2)
\approx_{d,\gamma}
\int_{\Rd}|\widehat{\eta_2}(\xi)|^2|\xi|^{2\gamma-d}\,d\xi
=
\int_{\Rd}|\widehat\mu_2(\xi)|^4|\xi|^{2\gamma-d}\,d\xi,
\]
since $\widehat{\eta_2}(\xi)=|\widehat\mu_2(\xi)|^2$.

Since $\mu_2$ is a probability measure, $|\widehat\mu_2(\xi)|\le 1$ for all $\xi$.
Therefore,
\[
\int_{|\xi|\le 1}|\widehat\mu_2(\xi)|^4|\xi|^{2\gamma-d}\,d\xi
\le
\int_{|\xi|\le 1}|\xi|^{2\gamma-d}\,d\xi
\lesssim_{d,\gamma}1.
\]

It remains to estimate the contribution from $\{|\xi|>1\}$.

\smallskip
\noindent
\textbf{Case 1: $\frac12\le \theta_2\le 1$.}
Since $2/\theta_2\le 4$ and $0\le |\widehat\mu_2|\le 1$, we have
\[
|\widehat\mu_2(\xi)|^4\le |\widehat\mu_2(\xi)|^{2/\theta_2}.
\]
Also, since $\gamma<a(v)\le v/(2\theta_2)$ and $|\xi|>1$,
\[
|\xi|^{2\gamma-d}\le |\xi|^{v/\theta_2-d}.
\]
Hence,
\[
\int_{|\xi|>1}|\widehat\mu_2(\xi)|^4|\xi|^{2\gamma-d}\,d\xi
\le
\int_{|\xi|>1}|\widehat\mu_2(\xi)|^{2/\theta_2}|\xi|^{v/\theta_2-d}\,d\xi
\le
\J_{v,\theta_2}(\mu_2)^{1/\theta_2}.
\]

\smallskip
\noindent
\textbf{Case 2: $0<\theta_2<\frac12$.}
Set
\[
g(\xi):=
|\widehat\mu_2(\xi)|^{2/\theta_2}|\xi|^{v/\theta_2-d}.
\]
Then, for $|\xi|>1$,
\[
|\widehat\mu_2(\xi)|^4|\xi|^{2\gamma-d}
=
g(\xi)^{2\theta_2}\,|\xi|^{2\gamma-2v+d(2\theta_2-1)}.
\]
Apply H\"older's inequality with exponents
\[
p=\frac{1}{2\theta_2},
\qquad \text{and} \qquad 
p'=\frac{1}{1-2\theta_2}.
\]
Since $\gamma<a(v)\le v$, we have $\gamma<v$, and therefore
\[
\frac{2\gamma-2v+d(2\theta_2-1)}{1-2\theta_2}<-d.
\]
Hence,
\begin{align*}
\int_{|\xi|>1}|\widehat\mu_2(\xi)|^4|\xi|^{2\gamma-d}\,d\xi
&=
\int_{|\xi|>1}g(\xi)^{2\theta_2}|\xi|^{2\gamma-2v+d(2\theta_2-1)}\,d\xi\\
&\le
\Bigl(\int_{|\xi|>1}g(\xi)\,d\xi\Bigr)^{2\theta_2}
\Bigl(
\int_{|\xi|>1}
|\xi|^{\frac{2\gamma-2v+d(2\theta_2-1)}{1-2\theta_2}}\,d\xi
\Bigr)^{1-2\theta_2}\\
&\lesssim_{d,\gamma,\theta_2,v}
\Bigl(
\int_{\Rd}
|\widehat\mu_2(\xi)|^{2/\theta_2}|\xi|^{v/\theta_2-d}\,d\xi
\Bigr)^{2\theta_2}
=
\J_{v,\theta_2}(\mu_2)^2.
\end{align*}

Combining the estimates together, we obtain
\[
I_{2\gamma}(\eta_2)
\lesssim_{d,\gamma,\theta_2,v}
1+\J_{v,\theta_2}(\mu_2)^{2\rho(\theta_2)}.
\]
Since $\supp\mu_2\subset B(0,R)$ and $\mu_2$ is a probability measure,
\[
|\widehat\mu_2(\xi)-1|
=
\Bigl|
\int \bigl(e^{-2\pi i x\cdot \xi}-1\bigr)\,d\mu_2(x)
\Bigr|
\le
2\pi R|\xi|.
\]
Thus, if $|\xi|\le (4\pi R)^{-1}$, then $|\widehat\mu_2(\xi)|\ge \frac12$. Therefore
\[
\int_{\Rd}
|\widehat\mu_2(\xi)|^{2/\theta_2}|\xi|^{v/\theta_2-d}\,d\xi
\ge
2^{-2/\theta_2}\int_{|\xi|\le (4\pi R)^{-1}} |\xi|^{v/\theta_2-d}\,d\xi
\gtrsim_{v,\theta_2,d,R} 1,
\]
and hence
\[
\cJ_{v,\theta_2}(\mu_2)\gtrsim_{v,\theta_2,d,R} 1.
\]
So the additive constant can be absorbed. Taking square roots proves
\eqref{eq:eta-energy-general-theta}.
\end{proof}

We now derive the mixed decay estimate for $A(\tau)$.

\begin{proposition}\label{prop:A-decay-theta1-theta2}
Assume
\[
0<\theta_1\le 1,
\qquad
0<u<d\theta_1,
\qquad
0<\theta_2\le 1,
\qquad
v>0,
\]
and
\[
\J_{u,\theta_1}(\mu_1)<\infty,
\qquad
\J_{v,\theta_2}(\mu_2)<\infty.
\]
Set
\[
a(v):=\min\!\left\{\frac d2,\, v, \frac{v}{2\theta_2}\right\}.
\]
Then, for every $0<\gamma<a(v)$,
\begin{equation}\label{eq:A-decay-theta1-theta2}
A(\tau)
\lesssim_{d,u,\theta_1,\gamma,\theta_2,v,R}
\J_{u,\theta_1}(\mu_1)^{1/2}
\J_{v,\theta_2}(\mu_2)^{\theta_1\rho(\theta_2)}
(1+|\tau|)^{-\beta_\gamma(u)},
\end{equation}
where
\[
\beta_\gamma(u):=\frac{u+2\theta_1\gamma-d\theta_1}{2}.
\]
\end{proposition}

\begin{proof}
Set
\[
m:=\frac{u}{\theta_1}.
\]
Since $0<u<d\theta_1$, we have $0<m<d$.

First, \eqref{eq:fhat-theta1-theta2} and Fubini imply
\[
\|\widehat f_{\tau}\|_{L^1(\Rd)}
\le
\iint \|\widehat\chi(\cdot-2\tau(y-z))\|_{L^1(\Rd)}\,d\mu_2(y)\,d\mu_2(z)
=
\|\widehat\chi\|_{L^1(\Rd)}.
\]
Also, Lemma~\ref{lem:agenera} gives
\[
\int_{\Rd}|\widehat f_{\tau}(\xi)|^2|\xi|^{d-m}\,d\xi
\lesssim_{d,m,\gamma,R}
(1+|\tau|)^{d-m-2\gamma}I_{2\gamma}(\eta_2).
\]
Hence, the weighted $L^{2/(2-\theta_1)}$ quantity in
Lemma~\ref{lem:duality-general} is finite, and the lemma applies to $f_\tau$.

Applying Lemma~\ref{lem:duality-general} with $\theta=\theta_1$ and $f=f_\tau$,
we obtain
\[
A(\tau)
\le
\J_{u,\theta_1}(\mu_1)^{1/2}
\left(\int_{\Rd}|\widehat f_{\tau}(\xi)|^{\frac{2}{2-\theta_1}}
|\xi|^{\frac{\theta_1 d-u}{2-\theta_1}}\,d\xi\right)^{\frac{2-\theta_1}{2}}.
\]
If $0<\theta_1<1$, then H\"older's inequality gives
\begin{align*}
\int_{\Rd}|\widehat f_{\tau}(\xi)|^{\frac{2}{2-\theta_1}}
|\xi|^{\frac{\theta_1 d-u}{2-\theta_1}}\,d\xi
& =
\int_{\Rd}
|\widehat f_{\tau}(\xi)|^{\frac{2(1-\theta_1)}{2-\theta_1}}
\Bigl(|\widehat f_{\tau}(\xi)|^2|\xi|^{d-u/\theta_1}\Bigr)^{\frac{\theta_1}{2-\theta_1}}
\,d\xi\\
&\le
\|\widehat f_{\tau}\|_{L^1(\Rd)}^{\frac{2(1-\theta_1)}{2-\theta_1}}
\left(
\int_{\Rd}|\widehat f_{\tau}(\xi)|^2|\xi|^{d-u/\theta_1}\,d\xi
\right)^{\frac{\theta_1}{2-\theta_1}}.
\end{align*}
Raising both sides to the power $(2-\theta_1)/2$, we get
\begin{equation}\label{eq:interp-theta1-theta2}
\left(\int_{\Rd}|\widehat f_{\tau}(\xi)|^{\frac{2}{2-\theta_1}}
|\xi|^{\frac{\theta_1 d-u}{2-\theta_1}}\,d\xi\right)^{\frac{2-\theta_1}{2}}
\le
\|\widehat f_{\tau}\|_{L^1(\Rd)}^{1-\theta_1}
\left(
\int_{\Rd}|\widehat f_{\tau}(\xi)|^2|\xi|^{d-u/\theta_1}\,d\xi
\right)^{\theta_1/2}.
\end{equation}
If $\theta_1=1$, then \eqref{eq:interp-theta1-theta2} is immediate.

Combining the inequalities, we obtain
\[
A(\tau)
\lesssim_{d,u,\theta_1,\gamma,R}
\J_{u,\theta_1}(\mu_1)^{1/2}
I_{2\gamma}(\eta_2)^{\theta_1/2}
(1+|\tau|)^{\frac{\theta_1}{2}(d-u/\theta_1-2\gamma)}.
\]
Since
\[
\frac{\theta_1}{2}\left(d-\frac{u}{\theta_1}-2\gamma\right)
=
-\frac{u+2\theta_1\gamma-d\theta_1}{2}
=
-\beta_\gamma(u),
\]
and since Lemma~\ref{lem:eta-energy-general-theta} gives
\[
I_{2\gamma}(\eta_2)^{1/2}
\lesssim_{d,\gamma,\theta_2,v,R}
\J_{v,\theta_2}(\mu_2)^{\rho(\theta_2)},
\]
we arrive at \eqref{eq:A-decay-theta1-theta2}.
\end{proof}
We will also need the following elementary monotonicity property, which allows us to pass from \(\J_{u,\theta}\)-control to \(\J_{w,\eta}\)-control for suitable ranges of parameters.
\begin{lemma}\label{lem:Jtheta-to-Jeta}
Let \(0<\theta\le 1\), let \(\mu\) be a compactly supported probability measure on $\Rd$, and assume
\[
\J_{u,\theta}(\mu)<\infty
\]
for some \(u>0\).

\begin{enumerate}
\item If \(0<\eta\le \theta\), then
\[
\J_{w,\eta}(\mu)<\infty
\qquad\text{for every } 0<w\le \frac{\eta}{\theta}u.
\]

\item If \(\theta<\eta\le 1\), then
\[
\J_{w,\eta}(\mu)<\infty
\qquad\text{for every } 0<w<u.
\]
\end{enumerate}
\end{lemma}

\begin{proof}
Write
\[
g(\rho):=\dim_\mathrm{F}^\rho \mu
=
\sup\{s\ge 0:\J_{s,\rho}(\mu)<\infty\},
\qquad 0\le \rho\le 1.
\]
It follows from \cite[Theorems 1.1 and 1.3]{JMF} that $g$ is non-decreasing and concave on $[0, 1]$. 

Since \(\mu\) is finite, \(|\widehat{\mu}|\le \mu(\mathbb R^d)\), and therefore \(g(0)\ge 0\).

Assume \(\J_{u,\theta}(\mu)<\infty\). Then \(g(\theta)\ge u\).

We first suppose \(0<\eta\le \theta\). Since \(g\) is concave on \([0,1]\),
\[
g(\eta)\ge \frac{\eta}{\theta}g(\theta)+\left(1-\frac{\eta}{\theta}\right)g(0)
\ge \frac{\eta}{\theta}u.
\]
Hence
\[
\J_{w,\eta}(\mu)<\infty
\qquad \text{for every }0<w<\frac{\eta}{\theta}u.
\]
It remains to treat the endpoint \(w=\frac{\eta}{\theta}u\). Set \(M:=\mu(\mathbb R^d)\). Then
\[
\J_{\eta u/\theta,\eta}(\mu)^{1/\eta}
=
\int_{\mathbb R^d} |\widehat{\mu}(\xi)|^{2/\eta} |\xi|^{u/\theta-d}\,d\xi.
\]
We split the integral into \(|\xi|\le 1\) and \(|\xi|>1\).

On \(\{|\xi|\le 1\}\), using \(|\widehat{\mu}(\xi)|\le M\),
\[
\int_{|\xi|\le 1} |\widehat{\mu}(\xi)|^{2/\eta} |\xi|^{u/\theta-d}\,d\xi
\le
M^{2/\eta}\int_{|\xi|\le 1} |\xi|^{u/\theta-d}\,d\xi<\infty,
\]
since \(u/\theta>0\).

On \(\{|\xi|>1\}\), since \(\eta\le \theta\), we have \(2/\eta-2/\theta\ge 0\), and so
\[
|\widehat{\mu}(\xi)|^{2/\eta}
=
|\widehat{\mu}(\xi)|^{2/\theta}\,
|\widehat{\mu}(\xi)|^{2/\eta-2/\theta}
\le
M^{2/\eta-2/\theta}|\widehat{\mu}(\xi)|^{2/\theta}.
\]
Therefore,
\[
\int_{|\xi|>1} |\widehat{\mu}(\xi)|^{2/\eta} |\xi|^{u/\theta-d}\,d\xi
\le
M^{2/\eta-2/\theta}
\int_{|\xi|>1} |\widehat{\mu}(\xi)|^{2/\theta} |\xi|^{u/\theta-d}\,d\xi
<\infty,
\]
because \(\J_{u,\theta}(\mu)<\infty\). Thus, \(\J_{\eta u/\theta,\eta}(\mu)<\infty\), proving part (1).

We now move to the case \(\theta<\eta\le 1\). Since \(g\) is non-decreasing,
\[
g(\eta)\ge g(\theta)\ge u.
\]
Hence
\[
\J_{w,\eta}(\mu)<\infty
\qquad\text{for every }0<w<u.
\]
This proves part (2).
\end{proof}
We are now ready to prove Theorem \ref{thm-ref}.
\begin{proof}[Proof of Theorem \ref{thm-ref}]
For simplicity, let $s:=\dim_\textup{F}^{\theta_2} \mu_2$. For each $v\in(0,s)$, write
\[
a(v):=\min\!\left\{\frac d2,\,v,\,\frac{v}{2\theta_2}\right\}.
\]
Then $a(v)\uparrow a$ as $v\uparrow s$.

\smallskip
\noindent
\textbf{Case 1: $0<u\le d\theta_1$.}
Fix
\[
0<\sigma<\min\{1,\beta(u)\}.
\]
We may choose parameters
\[w<u, \qquad
0<v<s,
\qquad
0<\gamma<a(v),
\]
so that
\begin{equation}\label{eq:case1-sigma-choice}
\sigma<\frac{w+2\theta_1\gamma-d\theta_1}{2},
\end{equation}
since
\[
\frac{w+2\theta_1\gamma-d\theta_1}{2}
\longrightarrow
\frac{u+2\theta_1 a-d\theta_1}{2}=\beta(u),
\]
as $w\uparrow u$, $v\uparrow s$ and $\gamma\uparrow a(v)$.

By the definition of $\dim_\textup{F}^{\theta_1} \mu_1$ and $\dim_\textup{F}^{\theta_2} \mu_2$, we have
$
\J_{w,\theta_1}(\mu_1)<\infty,
$
and 
$
\J_{v,\theta_2}(\mu_2)<\infty.
$
Hence, Proposition~\ref{prop:A-decay-theta1-theta2} implies
\[
A(\tau)\lesssim (1+|\tau|)^{-\frac{w+2\theta_1\gamma-d\theta_1}{2}}.
\]
By \eqref{eq:case1-sigma-choice},
\[
\int_{\R}A(\tau)|\tau|^{\sigma-1}\,d\tau<\infty.
\]
By \eqref{eq:pinned-fourier-abs} and the definition \eqref{eq:A-def} of $A(\tau)$, we have, for every $\tau\in\R$,
\begin{equation}\label{eq:A-as-pinned-average}
A(\tau)=\int_{\Rd}\bigl|\widehat{\delta_{x,\mu_2}^2}(\tau)\bigr|^2\,d\mu_1(x).
\end{equation}
Thus, Tonelli's theorem gives
\[
\int_{\Rd}\int_{\R}\bigl|\widehat{\delta_{x,\mu_2}^2}(\tau)\bigr|^2|\tau|^{\sigma-1}\,d\tau\,d\mu_1(x)<\infty.
\]
Therefore, for $\mu_1$-almost every $x\in\Rd$,
\[
\int_{\R}\bigl|\widehat{\delta_{x,\mu_2}^2}(\tau)\bigr|^2|\tau|^{\sigma-1}\,d\tau<\infty,
\]
and Lemma~\ref{lem:pinned-distance-measure-fourier} implies that
\[
\hd D_x(\supp\mu_2)\ge \sigma
\qquad\text{for $\mu_1$-almost every }x\in\Rd.
\]
Since $\sigma<\min\{1,\beta(u)\}$ was arbitrary, we conclude that
\[
\hd D_x(\supp\mu_2)\ge \min\{1,\beta(u)\}
\qquad\text{for $\mu_1$-almost every }x\in\Rd.
\]

If $\beta(u)>1$, choose $v\in(0,s)$ and $\gamma\in(0,a(v))$ so that
\[
\frac{u+2\theta_1\gamma-d\theta_1}{2}>1.
\]
Then, Proposition~\ref{prop:A-decay-theta1-theta2} gives $A\in L^1(\R)$, and \eqref{eq:A-as-pinned-average} implies
\[
\int_{\Rd}\int_{\R}\bigl|\widehat{\delta_{x,\mu_2}^2}(\tau)\bigr|^2\,d\tau\,d\mu_1(x)<\infty.
\]
Hence, for $\mu_1$-almost every $x\in\Rd$,
\[
\int_{\R}\bigl|\widehat{\delta_{x,\mu_2}^2}(\tau)\bigr|^2\,d\tau<\infty,
\]
and Lemma~\ref{lem:pinned-distance-measure-fourier} yields
\[
\cL^1\bigl(D_x(\supp\mu_2)\bigr)>0
\qquad\text{for $\mu_1$-almost every }x\in\Rd.
\]
\textbf{Case 2: $d\theta_1<u\le d$.}
Fix
\[
0<\sigma<\min\left\{1,\frac{ua}{d}\right\}.
\]
Choose $u_0\in(d\theta_1,u)$ and $v\in(0,s)$ so close to $u$ and $s$, respectively, that
\begin{equation}\label{eq:case2-u0-choice}
\sigma<\frac{u_0a(v)}{d}.
\end{equation}
Set
\[
\eta:=\frac{u_0}{d}.
\]
Then $\theta_1<\eta\le 1$, and since $u_0<u=\dim_\textup{F}^{\theta_1} \mu_1$, we have
\[
\J_{u_0,\theta_1}(\mu_1)<\infty.
\]
Part~(2) of Lemma~\ref{lem:Jtheta-to-Jeta} therefore implies
\[
\J_{w,\eta}(\mu_1)<\infty
\qquad\text{for every }0<w<u_0.
\]
Because $d\eta=u_0$, the exponent
\[
\frac{w+2\eta\gamma-d\eta}{2}
=
\frac{w+2(u_0/d)\gamma-u_0}{2}
\]
converges to $u_0a(v)/d$ as $w\uparrow u_0$ and $\gamma\uparrow a(v)$. In view of
\eqref{eq:case2-u0-choice}, we may choose
\[
0<w<u_0,
\qquad
0<\gamma<a(v),
\]
so that
\begin{equation}\label{eq:case2-sigma-choice}
\sigma<\frac{w+2\eta\gamma-d\eta}{2}.
\end{equation}
Applying Proposition~\ref{prop:A-decay-theta1-theta2} with $(w,\eta)$ in place of $(u,\theta_1)$ and with the same $v$ gives
\[
A(\tau)\lesssim (1+|\tau|)^{-\frac{w+2\eta\gamma-d\eta}{2}}.
\]
By \eqref{eq:case2-sigma-choice},
\[
\int_{\R}A(\tau)|\tau|^{\sigma-1}\,d\tau<\infty.
\]
Using \eqref{eq:A-as-pinned-average} and Tonelli once again, we obtain
\[
\int_{\Rd}\int_{\R}\bigl|\widehat{\delta_{x,\mu_2}^2}(\tau)\bigr|^2|\tau|^{\sigma-1}\,d\tau\,d\mu_1(x)<\infty.
\]
Hence, for $\mu_1$-almost every $x\in\Rd$,
\[
\int_{\R}\bigl|\widehat{\delta_{x,\mu_2}^2}(\tau)\bigr|^2|\tau|^{\sigma-1}\,d\tau<\infty,
\]
and Lemma~\ref{lem:pinned-distance-measure-fourier} gives
\[
\hd D_x(\supp\mu_2)\ge \sigma
\qquad\text{for $\mu_1$-almost every }x\in\Rd.
\]
Since $\sigma<\min\{1,ua/d\}$ was arbitrary, we conclude that
\[
\hd D_x(\supp\mu_2)\ge \min\left\{1,\frac{ua}{d}\right\}
\qquad\text{for $\mu_1$-almost every }x\in\Rd.
\]

If $ua/d>1$, choose $u_0\in(d\theta_1,u)$ and $v\in(0,s)$ such that
$u_0a(v)/d>1$. Define $\eta=u_0/d$ as above. Then choose $w<u_0$ and
$\gamma<a(v)$ so that
\[
\frac{w+2\eta\gamma-d\eta}{2}>1.
\]
Proposition~\ref{prop:A-decay-theta1-theta2} again gives $A\in L^1(\R)$, and the same Tonelli argument as above shows that, for $\mu_1$-almost every $x\in\Rd$,
\[
\int_{\R}\bigl|\widehat{\delta_{x,\mu_2}^2}(\tau)\bigr|^2\,d\tau<\infty.
\]
Lemma~\ref{lem:pinned-distance-measure-fourier} then implies
\[
\cL^1\bigl(D_x(\supp\mu_2)\bigr)>0
\qquad\text{for $\mu_1$-almost every }x\in\Rd.
\]
This completes the proof.
\end{proof}

\section{Proof of Theorem \ref{thm-theta-zero}} \label{sec:proof2}

To prove Theorem \ref{thm-theta-zero}, we make use of the following lemma.
\begin{lemma}\label{proptheta0}
Let \(d\ge 2\), and let \(\mu_1,\mu_2\) be compactly supported probability
measures on \(\Rd\). Assume that, for some \(s>0\),
\[
|\widehat{\mu_1}(\xi)| \lesssim \ang{\xi}^{-s/2}
\qquad (\xi\in\Rd).
\]
Then, for every \(0<\sigma<\min\{1,s/2\}\) such that \(I_\sigma(\mu_2)<\infty\),
\[
\hd D_x(\spt \mu_2)\ge \sigma
\qquad \text{for $\mu_1$-almost all }x\in\Rd.
\]
Moreover, if \(s>2\) and \(I_1(\mu_2)<\infty\), then
\[
\cL^1\bigl(D_x(\spt \mu_2)\bigr)>0
\qquad \text{for $\mu_1$-almost all }x\in\Rd.
\]
\end{lemma}
\begin{proof}
For \(\tau\in\R\), recall
\[
A(\tau)=\int_{\Rd}\bigl|\widehat{\delta_{x,\mu_2}^2}(\tau)\bigr|^2\,d\mu_1(x).
\]
By Lemma~\ref{lem:pinned-distance-measure-fourier}, Fubini's theorem, and the identity
\[
|x-y|^2-|x-z|^2 = |y|^2-|z|^2 - 2x\cdot (y-z),
\]
we obtain
\begin{align}
A(\tau)
&=
\iiint
e^{-2\pi i\tau(|x-y|^2-|x-z|^2)}
\,d\mu_1(x)\,d\mu_2(y)\,d\mu_2(z) \notag\\
&=
\iint
e^{-2\pi i\tau(|y|^2-|z|^2)}
\widehat{\mu_1}\bigl(-2\tau(y-z)\bigr)
\,d\mu_2(y)\,d\mu_2(z). \label{eq:theta-zero-A}
\end{align}
Hence
\[
A(\tau)
\lesssim
\iint
(1+2|\tau||y-z|)^{-s/2}
\,d\mu_2(y)\,d\mu_2(z).
\]

Fix \(0<\sigma<\min\{1,s/2\}\), and assume \(I_\sigma(\mu_2)<\infty\).
Since the integrand is non-negative, Tonelli's theorem gives
\begin{align*}
\int_{\R} A(\tau)|\tau|^{\sigma-1}\,d\tau
&\lesssim
\iint
\left(
\int_{\R}
(1+2|\tau||y-z|)^{-s/2}
|\tau|^{\sigma-1}\,d\tau
\right)
d\mu_2(y)\,d\mu_2(z).
\end{align*}
For \(r=|y-z|>0\), the change of variables \(u=2|\tau|r\) yields
\[
\int_{\R}
(1+2|\tau|r)^{-s/2}|\tau|^{\sigma-1}\,d\tau
=
c_{\sigma,s}\,r^{-\sigma},
\]
where
\[
c_{\sigma,s}
=
2^{1-\sigma}\int_0^\infty (1+u)^{-s/2}u^{\sigma-1}\,du
<\infty,
\]
since \(0<\sigma<s/2\). Therefore
\[
\int_{\R} A(\tau)|\tau|^{\sigma-1}\,d\tau
\lesssim
\iint |y-z|^{-\sigma}\,d\mu_2(y)\,d\mu_2(z)
=
I_\sigma(\mu_2)
<\infty.
\]
It follows from Tonelli's theorem again that
\[
\int_{\Rd}\int_{\R}
\bigl|\widehat{\delta_{x,\mu_2}^2}(\tau)\bigr|^2
|\tau|^{\sigma-1}\,d\tau\,d\mu_1(x)
<\infty.
\]
Hence, for \(\mu_1\)-almost every \(x\in\Rd\),
\[
\int_{\R}
\bigl|\widehat{\delta_{x,\mu_2}^2}(\tau)\bigr|^2
|\tau|^{\sigma-1}\,d\tau
<\infty.
\]
Lemma~\ref{lem:pinned-distance-measure-fourier} then implies
\[
\hd D_x(\spt \mu_2)\ge \sigma
\qquad\text{for $\mu_1$-almost all }x\in\Rd.
\]

For the positive Lebesgue measure statement, assume \(s>2\) and \(I_1(\mu_2)<\infty\).
By the same argument, but without the weight \(|\tau|^{\sigma-1}\), we obtain
\begin{align*}
\int_{\R} A(\tau)\,d\tau
&\lesssim
\iint
\left(
\int_{\R}(1+2|\tau||y-z|)^{-s/2}\,d\tau
\right)
d\mu_2(y)\,d\mu_2(z) \\
&=
C_s
\iint |y-z|^{-1}\,d\mu_2(y)\,d\mu_2(z)
=
C_s I_1(\mu_2)
<\infty,
\end{align*}
where
\[
C_s = \int_{\R}(1+|u|)^{-s/2}\,du < \infty.
\]
Therefore,
\[
\int_{\Rd}\int_{\R}
\bigl|\widehat{\delta_{x,\mu_2}^2}(\tau)\bigr|^2
\,d\tau\,d\mu_1(x)
<\infty,
\]
so, for \(\mu_1\)-almost every \(x\in\Rd\),
\[
\int_{\R}
\bigl|\widehat{\delta_{x,\mu_2}^2}(\tau)\bigr|^2
\,d\tau
<\infty.
\]
Lemma~\ref{lem:pinned-distance-measure-fourier} implies
\[
\cL^1\bigl(D_x(\spt \mu_2)\bigr)>0
\qquad\text{for $\mu_1$-almost all }x\in\Rd.
\]
This completes the proof.
\end{proof}

\begin{proof}[Proof of Theorem \ref{thm-theta-zero}]
We first prove (1). If \(\fd\mu=0\), then the conclusion is trivial. So assume
\(\fd\mu>0\).

Let \((\sigma_n)_{n\ge 1}\) be an increasing sequence in
\[
\left(0,\min\left\{1,\frac{\fd\mu}{2}\right\}\right)
\]
such that
\[
\sigma_n\uparrow \min\left\{\frac{\fd\mu}{2},1\right\}.
\]
For each \(n\), choose \(t_n\) such that
\[
2\sigma_n<t_n<\fd\mu.
\]
Since \(t_n<\fd\mu\), by the definition of Fourier dimension,
\[
|\widehat{\mu}(\xi)|\lesssim \langle \xi\rangle^{-t_n/2}
\qquad (\xi\in\Rd).
\]
Moreover,
\[
I_{\sigma_n}(\mu)
\approx
\int_{\Rd} |\widehat{\mu}(\xi)|^2 |\xi|^{\sigma_n-d}\,d\xi
\lesssim
\int_{\Rd} \langle \xi\rangle^{-t_n} |\xi|^{\sigma_n-d}\,d\xi
<\infty,
\]
because \(\sigma_n<t_n\).

Therefore, Lemma~\ref{proptheta0}, applied with
\(\mu_1=\mu_2=\mu\) and \(s=t_n\), yields a set \(G_n\subseteq \spt\mu\) such that
\[
\mu(G_n)=1
\]
and
\[
\hd D_x(\spt\mu)\ge \sigma_n
\qquad\text{for all }x\in G_n.
\]
Now let
\[
G=\bigcap_{n=1}^\infty G_n.
\]
Then \(\mu(G)=1\). Also, for every \(x\in G\) and every \(n\),
\[
\hd D_x(\spt\mu)\ge \sigma_n.
\]
Letting \(n\to\infty\), we conclude that
\[
\hd D_x(\spt\mu)\ge \min\left\{\frac{\fd\mu}{2},1\right\}
\qquad\text{for all }x\in G.
\]
This proves the first claim.

To prove the second claim in (1), assume that \(\fd\mu>2\). Choose \(t\) such that
\[
2<t<\fd\mu.
\]
Then, by the definition of Fourier dimension,
\[
|\widehat{\mu}(\xi)|\lesssim \langle \xi\rangle^{-t/2}
\qquad (\xi\in\Rd).
\]
Since \(1<t\), we have
\[
I_1(\mu)
\approx
\int_{\Rd} |\widehat{\mu}(\xi)|^2 |\xi|^{1-d}\,d\xi
\lesssim
\int_{\Rd} \langle \xi\rangle^{-t} |\xi|^{1-d}\,d\xi
<\infty.
\]
Applying Lemma~\ref{proptheta0} with \(\mu_1=\mu_2=\mu\), we obtain
\[
\cL^1\bigl(D_x(\spt\mu)\bigr)>0
\qquad\text{for \(\mu\)-almost all }x\in\Rd.
\]

We now prove (2). Recalling that $\mu$ is compactly supported, by \cite[Proposition~4.2]{restriction2},
\[
\fs\mu \le \fd\mu + d\theta
\]
and therefore
\[
\fd\mu \ge \fs\mu-d\theta.
\]
Applying (1), we obtain
\[
\hd D_x(\spt\mu)\ge \min\left\{\frac{\fd\mu}{2},1\right\}
\ge
\min\left\{\frac{\fs\mu-d\theta}{2},1\right\}
\qquad\text{for \(\mu\)-almost all }x\in\Rd.
\]
This proves the first claim in (2).

If \(\fs\mu > 2+d\theta\), then $\fd \mu>2$ and applying (1) 
\[
\mathcal{L}^1\left(D_x(\spt\mu)\right)>0 
\qquad\text{for \(\mu\)-almost all }x\in\Rd.
\]
This completes the proof.
\end{proof}

\end{document}